\documentclass[11pt]{amsart}
\usepackage{slashbox}
\usepackage{hyperref,amsmath,amsthm,amsfonts,amscd,flafter,epsf,amssymb,ifsym,wasysym}
\usepackage{epsfig,graphicx,color}
\usepackage{amsfonts}
\usepackage{amssymb}
\usepackage{amsmath}
\usepackage{amsthm}
\usepackage{latexsym}
\usepackage{amscd}
\usepackage{epsf}

\newcommand{\SCC}{stable weighted graph }

\newtheorem{theorem}{Theorem}
\newtheorem{thm}{Theorem}[section]
\newtheorem{prop}[thm]{Proposition}
\newtheorem{cor}[thm]{Corollary}

\newtheorem{lem}[thm]{Lemma}

\newtheorem{defn}[thm]{Definition}

\newtheorem{remark}[thm]{Remark}

\def\endproof{\relax\ifmmode\expandafter\endproofmath\else
  \unskip\nobreak\hfil\penalty50\hskip.75em\hbox{}\nobreak\hfil\bull
  {\parfillskip=0pt \finalhyphendemerits=0 \bigbreak}\fi}
\def\endproofmath$${\eqno\bull$$\bigbreak}
\def\bull{\vbox{\hrule\hbox{\vrule\kern3pt\vbox{\kern6pt}\kern3pt\vrule}\hrule}}

\newcommand{\Aut}{\mathrm{Aut}}

\newcommand{\Q}{\mathbb{Q}}
\newcommand{\PP}{\mathrm{P}}

\newcommand{\Z}{\mathbb{Z}}
\newcommand{\Ccal}{\mathcal{C}}
\newcommand{\Acal}{\mathcal{A}}

\newcommand{\QQ}{\mathrm{Q}}
\newcommand{\sig}{\sigma}
\newcommand{\Pcal}{\mathcal{P}}
\newcommand{\ov}{\widehat}
\newcommand{\Det}{\mathrm{Det}}
\newcommand{\rank}{\mathrm{rank}}
\newcommand{\z}{\mathrm{\bf{z}}}
\newcommand{\w}{\mathrm{\bf{w}}}
\newcommand{\p}{\mathrm{\bf{q}}}
\newcommand{\q}{\mathrm{\bf{p}}}

\newcommand{\ka}{\kappa}
\newcommand{\kring}{\ka^*\left(\Modbar_{g,n}\right)}
\newcommand{\kringc}{\ka^*_c\left(\Modbar_{g,n}\right)}
\newcommand{\Mgnbar}{\Modbar_{g,n}}
\newcommand{\Mgn}{\Mod_{g,n}}

\newcommand{\lra}{\longrightarrow}
\newcommand{\ra}{\rightarrow}
\newcommand{\Sig}{\Sigma}

\newcommand{\Mod}{\mathcal{M}}
\newcommand{\Modbar}{{\overline{\mathcal{M}}}}

\newcommand{\Ker}{\mathrm{Ker}}

\begin{document}

\title{On the structure of the kappa-ring}%
\author{Eaman Eftekhary}%
\address{School of Mathematics, Institute for Research in Fundamental Science (IPM),
P. O. Box 19395-5746, Tehran, Iran}%
\email{eaman@ipm.ir}
\author{Iman Setayesh}%
\address{School of Mathematics, Institute for Research in Fundamental Science (IPM),
P. O. Box 19395-5746, Tehran, Iran}%
\email{setayesh@ipm.ir}

% ----------------------------------------------------------------
\begin{abstract}
We obtain  lower bounds on the
 rank of the kappa ring $\kring$ of the Delign-Mumford 
compactification of the moduli space of curves in different degrees. For this purpose,
we introduce a quotient $\kappa_c^*(\Mgnbar)$ of $\kring$, and show that
the rank of this latter ring in degree $d$ is bounded 
below by  $|\PP(d,3g-2+n-d)|$ where $\PP(d,r)$ denotes the set of partitions 
of the positive integer $d$ into at most $r$ parts. In codimension 
$1$ (i.e. $d=3g-4+n$) we show that the rank of $\kappa_c^*(\Mgnbar)$ 
is equal to $n-1$ for $g=1$, and 
is  equal to $$\left\lceil \frac{(n+1)(g+1)}{2}\right\rceil-1$$
for $g>1$. Furthermore, in codimension $e=3g-3+n-d$, the rank of 
$\kappa_c^*(\Mgnbar)$ (as $g$ and $e$ remain fixed and $n$ grows large)
is asymptotic to 
$$\frac{{{n+e}\choose e}{{g+e}\choose e}}{(e+1)!}.$$
\end{abstract}
\maketitle
\section{Introduction}\label{sec:intro}
Let $\epsilon=\pi_{g,n}^1 : \overline{\mathcal{M}}_{g,n+1} \to 
\overline{\mathcal{M}}_{g,n}$ 
denote the universal curve over the moduli space $\Mgnbar$ 
of stable 
genus $g$,  $n$-pointed curves. Throughout this paper, we will assume that 
$n>0$. The psi and kappa classes in the Chow ring of
 $\overline{\mathcal{M}}_{g,n}$ are defined as follows~\cite{M}.
Let $\mathbb{L}_{i} \to \overline{\mathcal{M}}_{g,n+1}$ 
denote the cotangent line bundle
 over $\Modbar_{g,n+1}$ whose fiber over a given point (which is an $(n+1)$-pointed 
 genus $g$ curve) is the cotangent space over the $i^{th}$ marked point. 
 The $i^{th}$ psi-class $\psi_i$ 
 over $\overline{\mathcal{M}}_{g,n+1}$ is then defined by
$$\psi_i= c_1(\mathbb{L}_{i})\in A^{1}(\overline{\mathcal{M}}_{g,n+1}).$$
Correspondingly, the $i$-th kappa-class $\kappa_i$ is defined via
$$\kappa_i = \epsilon_{*}(\psi_{n+1}^{i+1}) \in A^{i}(\overline{\mathcal{M}}_{g,n}).$$
The push forwards of the $\kappa$ and $\psi$ classes from the strata generates the 
tautological ring $R^*(\overline{\mathcal{M}}_{g,n})$ \cite{F,GP}.\\

The $\kappa$, $\psi$ and tautological classes over an open subset $\mathcal{U}$
of the moduli space  $\Mgnbar$, and in particular over  the smooth part $\Mgn$ and 
the moduli $\Mgn^c$ of curves 
of compact type, are defined by restricting the respective classes from 
$\Mgnbar$ to the corresponding open subset $\mathcal{U}\subset \Mgnbar$. 
The kappa ring $\kring$ of the moduli space of curves is  the subring of the 
tautological ring $R^*(\Mgnbar)$ generated by the kappa classes 
$\kappa_1,\kappa_2,...$ over $\Q$. One may define the kappa ring 
$\kappa^*(\mathcal{U})$ in a similar way.
It was observed by Pandharipande~\cite{Rahul-k1,Rahul-k2}
that if one restricts attention to the moduli space $\Mgn^c$, 
the structure of the kappa ring may be completely determined 
using a combination of combinatorial arguments and localization ideas.
Pandharipande shows, consequently, that the rank of the ring 
$\kappa^*(\Mgn^c)$ in degree $d$ is equal to $|\PP(d,2g-2+n-d)|$, where 
$\PP(d,r)$ denotes the set of partitions of $d$ into at most $r$ parts.\\

We will use some combinatorial observations to obtain a number of
 lower bounds on the rank of $\kring$. For this purpose,
associated with every stable weighted graph (see the second section for 
the definition) one may define a natural cycle in the Chow ring 
of  $\Mgnbar$. Correspondingly, we may define 
the {\emph{combinatorial tautological ring}} to be the quotient of the tautological 
ring obtained by setting trivial the classes which integrate trivially over all such  
{\emph{combinatorial cycles}}. It also makes sense to talk about the combinatorial
kappa rings. These will be denoted by $R^*_c(\Mgnbar)$ and 
$\kappa^*_c(\Mgnbar)$, respectively.
It is then implied from the work of Pandharipande that 
$\kappa^*_c(\Mgn^c)=\kappa^*(\Mgn^c)$. \\

%In fact, we make the following 
%conjecture.
%\begin{conj}
%For any $g\geq 0$ and $n>0$ we have $\kappa^*_c(\Mgnbar)=\kring$.
%\end{conj}
Since $\kappa^*_c(\Mgnbar)$ is naturally a quotient of $\kring$, the rank of the 
former gives a lower bound on the rank of the latter. In this paper, we prove 
a number of theorems about the rank of $\kappa_c^*(\Mgnbar)$ in different degrees.
In particular we show
\begin{theorem}\label{thm:main-2}
The rank of the combinatorial kappa 
ring $\kappa^*_c(\Mgnbar)$  in degree $d$ is bounded below by $|\PP(d,3g-2+n-d)|$.
\end{theorem}
In codimension $1$, $\kappa^*_c(\Mgnbar)$ is relatively easy to describe,
as presented in the following theorem.

\begin{theorem} \label{thm:main-3}
For $d=3g-4+n$ and $g>1$
the rank of $\kappa^{d}_c(\Mgnbar)$ is equal to
\begin{displaymath}
\left\lceil \frac{(n+1)(g+1)}{2}\right\rceil -1,
\end{displaymath}
while for $g=1$ the rank is equal to $n-1$. 
\end{theorem}
 
The reason for the difference between genus $1$ and higher genus is that there 
is a relation between the combinatorial divisors of $\Modbar_{1,n}$ as far as 
the kappa classes are concerned. More precisely, let 
$D_i\simeq \Modbar_{1,n-i}\times \Modbar_{1,i+2},\ i=1,...,n-1$ 
denote the divisor in  $\Modbar_{1,n}$ which corresponds to a degeneration of an
$n$-pointed genus $1$ curve to an $(n-i)$-pointed genus $1$ curve and an 
$(i+2)$-pointed genus $0$ curve. Let $D_n\simeq \Modbar_{0,n+2}$ denote the
the divisor which corresponds to a degeneration of   an
$n$-pointed genus $1$ curve to  an $(n+2)$-pointed genus $0$ curve. 

\begin{theorem}\label{thm:main-4}
For every element $\kappa\in\kappa^{n-1}(\Modbar_{1,n})$
\begin{displaymath}
\frac{1}{12}\int_{[D_n]}\kappa=\sum_{i=1}^{n-1}{{n-2}\choose i-1}\int_{[D_i]}\kappa.
\end{displaymath}
\end{theorem}

The main tool for proving the above relation is an explicit formula for 
the integration of $\psi$ classes over $\Modbar_{1,n}$. Let $\sig_i(a_1,...,a_n)$ 
denote the $i$-th symmetric function in variables $a_1,...,a_n$. The formula is 
given by the following theorem.

\begin{theorem}\label{thm:main-5}
Suppose that the non-negative integers $a_1,...,a_n$ are given so that 
$a_1+...+a_n=n$. Then
\begin{displaymath}
\int_{\Modbar_{1,n}}\psi_1^{a_1}\hdots\psi_n^{a_n}=\frac{1}{24}{n\choose {a_1,\hdots,a_n}}
\left(1-\sum_{i=2}^n\frac{\sig_i(a_1,\hdots,a_n)}{i(i-1){n\choose i}}\right).
\end{displaymath}
\end{theorem}
 
When the codimension $e=3g-3+n-d$ is arbitrary, a similar theorem may be 
proved for the asymptotic behaviour of the rank of the combinatorial kappa ring 
$\kappa_c^d(\Mgnbar)$, as the number $n$ of the marked points grows large: 
\begin{theorem}\label{thm:main-6}
The  rank of 
the combinatorial kappa ring $\kappa_c(\Mgnbar)$ in codimension $e$, as the number 
$n$ of the marked points becomes large,
is asymptotic to 
$$\frac{{n+e \choose e}.{g+e\choose e}}{(e+1)!}.$$
\end{theorem}

The paper is organized as follows. In Section 2 we introduce the combinatorial 
cycles and discuss the integration of the $\kappa$ and $\psi$ classes over them. 
Each $d$-dimensional combinatorial cycle $\Ccal\subset \Mgnbar$ gives 
a linear map $\int_\Ccal:\kappa^d(\Mgnbar)\ra \Q$. 
Section 3 is (naively speaking) devoted to finding a list of combinatorial cycles
$\Ccal_1,...,\Ccal_N$ such that there are no {\emph{trivial}} relations among $\int_{\Ccal_i}$.
This gives a $N\times |\PP(d)|$ matrix $R(d;g,n)$, and the rank $r(d;g,n)$ of 
this matrix gives 
the rank of the combinatorial kappa ring in degree $d$. In Section 4 a strategy 
for estimating $r(d;g,n)$ from below is described. A quick corollary of the aforementioned 
strategy is Theorem~\ref{thm:main-2}. \\

In Section 5 we study the kappa ring in codimension $1$, which corresponds to 
$d=3g-4+n$. Explicit combinatorial formulas for the 
 integration of $\kappa$ and $\psi$ classes 
over $\Modbar_{1,n}$ (Theorem~\ref{thm:main-5}) are used to prove that 
there is a non-trivial relation among $\int_{\Ccal_i}$ when $g=1$ and $d=n-1$,
which is given by Theorem~\ref{thm:main-4}. Furthermore, using Theorem~\ref{thm:main-5} 
we show that the rank of $\kappa_c^{n-1}(\Modbar_{1,n})$ is $n-1$.
For $g>1$ the computation of the rank of $\kappa_c^{3g-4+n}(\Mgnbar)$ is reduced,
by the results of Section 4, to the computation of the rank of certain $(g+1)\times (g+1)$
matrices with entries consisting of the integrals of $\psi$ classes. We study
the aforementioned matrices, compute their determinant using the KdV, String and 
Dilation  equations, and conclude that all of them are full rank. The results are used 
 to prove Theorem~\ref{thm:main-3}. The same matrices re-appear as we 
study the asymptotic behaviour of the rank of $\kappa_c^{3g-3+n-e}(\Mgnbar)$ as
$n$ goes to infinity in Section 6. Once again, the results of Section 5 are used to 
prove Theorem~\ref{thm:main-6}.

{\bf{Acknowledgement.}} We would like to thank Rahul Pandharipande for 
several helpful discussions. We are also grateful to Carl Faber for pointing a 
mistake in the original version of this paper, and for suggesting the study 
of the kappa ring in codimension $1$. Many of the computations 
of this paper are performed on IPM's Computer Cluster, with the kind assistance 
of M. Ashtiani and E. Pourjani, and we would like to present our gratitude for their help.\\

\section{Combinatorial cycles and the $\psi$ classes}
Let us first consider an alternative basis for the kappa ring of
$\Mgnbar$, instead of the kappa classes. Let
$$\pi_{g,n}^{k}:\Modbar_{g,n+k}\ra \Mgnbar$$
denote the forgetful map which forgets the last
$k$ marked points.
\begin{defn}\label{defn:psi-classes}
For every multi-set $\q=(a_1\geq a_2\geq ...\geq a_k)$ of positive integers
define
\begin{itemize}
\item $|\q|:=k$ and $d(\q):=\sum_{i=1}^k a_i$
\item $\psi(\q):=\psi(a_1,...,a_{k}):=\left(\pi_{g,n}^{k}\right)_*
\left(\prod_{i=1}^k\psi_{n+i}^{a_i+1}\right)\in\kring$
\item $\kappa(\q)=\kappa(a_1,...,a_k):=\prod_{i=1}^k\kappa_{a_i}\in\kring$.
\end{itemize}
%\begin{displaymath}
%\begin{split}
%&|\q|:=k,\ \ \ d(\q):=\sum_{i=1}^k a_i,\\
%&\psi(\q):=\psi(a_1,...,a_{k}):=\left(\pi_{g,n}^{k}\right)_*
%\left(\prod_{i=1}^k\psi_{n+i}^{a_i+1}\right)\in\kring,\ \ \ \&\\
%&\kappa(\q)=\kappa(a_1,...,a_k):=\prod_{i=1}^k\kappa_{a_i}\in\kring.
%\end{split}
%\end{displaymath}
\end{defn}
The classes $\psi(\q)$, for $\q$ a multi-set  corresponding to a
partition of $d=d(\q)$ in $\PP(d)$, generate the
kappa ring $\kring$ in degree $d$ as a module over $\Q$. In particular,
for every positive integer $d$, $\kappa_d=\psi(d)$ is the $\psi$
class corresponding to the multi-set consisting of the single element
$d$. In order to describe the full correspondence between $\kappa$ 
and $\psi$ classes,
let $\q=(a_1,a_2,...,a_k)$ be a multi-set.
 For $\sigma\in S_k$ a permutation over $k$ elements let
 $$\sigma=\tau_1\tau_2...\tau_r,$$
be the canonical cycle decomposition for $\sigma$, including the $1$-cycles.
Let $\sigma(\q)_i$, for $i=1,2,...,r$, denote the sum of the elements of
$\q$ whose indices  correspond to the $i$-th cycle $\tau_i$. Define
$$\kappa_{\sigma(\q)}:=\prod_{i=1}^r\kappa_{\sigma(\q)_i}.$$
The following general formula is due to Faber and is discussed in \cite{AC}.
\begin{lem}\label{lem:Faber-formula}
For every partition $\q\in \PP(d)$ as above we have
$$\psi(\q)=\sum_{\sigma\in S_k}\kappa_{\sigma(\q)}.$$
\end{lem}
An immediate consequence of the above formula is the following lemma, 
c.f. lemma 1 in \cite{Rahul-k2}.
\begin{lem}in $\Acal^d(\Mgnbar)$ defined by
$$\Big\{ \psi(\q)\ \big|\ \q\in\PP(d)\Big\},\ \ \ \
\Big\{ \kappa(\q)\ \big|\ \q\in\PP(d)\Big\}$$
are related by an invertible linear transformation independent of $g$ and $n$.
\end{lem}
\begin{proof}
The transformation is given by Faber's formula from 
lemma~\ref{lem:Faber-formula} above.
In the partial ordering of $\PP(d)$ by length the transformation is 
triangular, with $1$'s on the diagonal, hence invertible.
\end{proof}

Our main tool for proving the independence of the generators in the kappa ring
is the integration against  the combinatorial cycles.

\begin{defn}
A {\emph{weighted graph}} $G$ is a finite connected 
graph with a set $V(G)$ of vertices and
a set $E(G)$ of edges, and a weight function
$$\epsilon=\epsilon_G:V(G)\ra \Z^{\geq 0}\times \Z^{\geq 0}.$$
For $i\in V(G)$ denote the degree of $i$by $d_i=d(i)$ and let $\epsilon(i)=(g_i,n_i)$.
A {\emph{\SCC}} is a weighted graph $G$
with the property that  for every vertex $i\in V(G)$, $2g_i+n_i+d_i>2$.
If $G$ is a \SCC define
\begin{itemize}
\item $g(G):=\left(\sum_{i\in V(G)}g_i\right)+|E(G)|-|V(G)|+1$
\item $n(G):=\sum_{i\in V(G)}n_i$.
\end{itemize}
%\begin{displaymath}
%\begin{split}
%g(G)&:=\left(\sum_{i\in V(G)}g_i\right)+|E(G)|-|V(G)|+1\\
%n(G)&:=\sum_{i\in V(G)}n_i.
%\end{split}
%\end{displaymath}
\end{defn}
Suppose that $G$ is a \SCC as above. An automorphism $\phi$ 
of $G$ consists of a pair of bijective 
maps  $\phi_V:V(G)\ra V(G)$ and $\phi_E:E(G)\ra E(G)$ and a subset 
$E'\subset E(G)$ such that the following are satisfied. 
\begin{itemize}
\item For every edge $e\in E(G)$ connecting the vertices $i,j\in V(G)$, $\phi_E(e)$ is an edge 
connecting $\phi_V(i)$ and $\phi_V(j)$. 
\item For every $i\in V(G)$,  $\epsilon(i)=\epsilon(\phi(i))$. 
\item Every $e\in E'$ connects a vertex to itself.
\end{itemize}
The automorphisms of $G$ form a group which  is denoted by $\Aut(G)$. \\

Associated with a \SCC $G$
we may construct a cycle in the Chow ring of $\Modbar_{g(G),n(G)}$
as follows. Associated with a vertex $i\in V(G)$, let
$\Modbar(i)$ denote the moduli space $\Modbar_{g_i,n_i+d_i}$ where the labels of the last
$d_i$ marked points correspond to those edges in $E(G)$ which are adjacent to
$i$. Let $\Ccal(G)$ denote the product
$$\Ccal(G)=\prod_{i\in V(G)}\Modbar(i)=\prod_{i\in V(G)}\Modbar_{g_i,n_i+d_i}.$$
Any automorphism $\phi$ of of the \SCC $G$ gives an automorphism of 
the space $\Ccal(G)$. 
The group $\Aut(G)$ of automorphisms of $G$ thus acts on $\Ccal(G)$. 
There is a map $\imath_G$ from the product $\Ccal(G)$ to the 
moduli space  $\Modbar_{g(G),n(G)}$, which is defined as follows.
Choose a point $(\Sig_i,\z_i\cup \w_i)_{i\in V(G)}$  of $\Ccal(G)$, where
$\z_i$ and $\w_i$ are sets of $n_i$ and $d_i$ marked points on the curve 
$\Sig_i$ of genus $g_i$ respectively.
For an edge $e\in E(G)$ connecting $i,j\in V(G)$ glue the marked points in $\w_i$
and $\w_j$ corresponding to $e$ to each other. The result is a stable curve
$$\imath_G\left(\left(\Sig_i, \z_i\cup \w_i\right)_{i\in V(G)}\right)
\in \Modbar_{g(G),n(G)}.$$
The map $\imath_G$ respects the action of $\Aut(G)$, and gives an 
embedding of $\Ccal(G)/\Aut(G)$ in $\Modbar_{g(G),n(G)}$.
Thus, a \SCC $G$ determines a cycle
$$\left(\imath_G\right)_* \left[{\Ccal(G)}\right]=
{|\Aut(G)|}.{\left(\imath_G\right)_* \left[\frac{\Ccal(G)}{\Aut(G)}\right]}
\in \Acal_{d}\left(\Modbar_{g(G),n(G)}\right),$$
where $d=3g(G)-3+n(G)-|E(G)|$. This cycle 
is denoted by $[G]$,  by slight abuse of notation.\\

 For a \SCC $G$ let us assume that $n=n(G)$ and $g=g(G)$. Let 
 $\psi(\q)=\psi(a_1,...,a_k)$ be a $\psi$ class. If
$$d(\q)+|E(G)|=3g-3+n$$
we may integrate $\psi(\q)$ over
$[G]\in \Acal_{3g-3+n-|\q|}(\Mgnbar)$. The integral
$$\Big\langle\psi(\q)\ ,\ [G]\Big\rangle=\int_{[G]}\psi(\q)
=\int_{(\pi^{k}_{g,n})^*[G]}\left(\prod_{j=1}^k \psi_{n+j}^{a_j+1}\right)\in \Q$$
may be computed in terms of the integrals of the top degree $\psi$ classes
over the moduli spaces $\Mgnbar$ as follows.
 The components of
$(\pi^{k}_{g,n})^*[G]$ are indexed by the functions
$$j:\big\{1,2,...,k\big\}\lra V(G).$$
Here $j(p)$ corresponds to  the component $\Modbar(j(p))$ of $\Ccal(G)$ which
contains the image of the $(n+p)$-th marked point under the forgetful map.
Let us denote the space of all such maps by $[k,G]$.
For $j\in[k,G]$ let us denote the corresponding component of $(\pi^{k}_{g,n})^*[G]$
by $[G,j]$.
% If $i\in V(G)$ denotes a vertex of $G$, 
%let us denote by $|i|_j$ the number of
%elements in $j^{-1}(i)$. Moreover, 
Define
$$\psi_i(\q,j):=\prod_{p\in j^{-1}(i)}\psi_{n+p}^{a_p+1},
\ \ \ \forall\ j\in[k,G],\ i\in V(G).$$
We may then compute
\begin{equation}\label{eq:int-over-cycles}
\begin{split}
\int_{(\pi^{k}_{g,n})^*[G]}\left(\prod_{j=1}^k \psi_{n+j}^{a_j+1}\right)&=
\sum_{j\in[k,G]}\int_{[G,j]}\left(\prod_{j=1}^k \psi_{n+j}^{a_j+1}\right)\\
&=\sum_{j\in[k,G]}\left(\prod_{i\in V(G)}
\int_{\left[\Modbar_{g_i,n_i+d_i+|j^{-1}(i)|}\right]}\psi_i(\q,j)\right).
\end{split}
\end{equation}

The degree of $\psi_i(\q,j)$ may be computed as
$$\deg\left(\psi_i(\q,j)\right)=\sum_{p\in j^{-1}(i)}(a_p+1)=
\left|j^{-1}(i)\right|+\sum_{p\in j^{-1}(i)}a_p %=:|i|_j+\sigma_i(\q,j)
.$$
The integrals in the last line of (\ref{eq:int-over-cycles})
are trivial unless the degree of $\psi_i(\q,j)$
matches the dimension of  $\Modbar_{g_i,n_i+d_i+|j^{-1}(i)|}$,
i.e. if and only if
$$\sum_{p\in j^{-1}(i)}a_p=3g_i-3+n_i+d_i.$$
Moreover, the value  computed in equation~\ref{eq:int-over-cycles}
does not depend on the graph $G$, and only depends on a  
modified weight multi-set. Let
$$Q=\Big\{(g,n)\in \Z^{\geq 0}\times \Z^+\ \big|\ 2g+n>2\Big\}.$$

If  $V(G)=\{1,...,m\}$, the modified weight multi-set associated with
$G$ is the multi-set
\begin{displaymath}
\begin{split}
&\p_G:=\Big(\theta_G(i)\in Q\ \big|\ i\in\{1,...,m\}\Big),\ \ \ \text{where}\\
&\theta_G(i):=(g_i,m_i=n_i+d_i),\ \ \ \forall\ 1\leq i\leq m.\\
\end{split}
\end{displaymath}
In other words, $\theta_G(i)$ records the genus and the total number of 
the marked points (i.e. marked points and nodes) on the component corresponding 
to the vertex $i$, and $\p_G$ records all the pairs corresponding to the vertices 
of $G$ (in a sense) regardless of the structure of $G$ as a graph.
We may assume that $3g_i-3+m_i>0$ for $i=1,...,k$ and $(g_i,m_i)=(0,3)$ for 
$i=k+1,...,m$.\\

Let $\q(\p)$ denote the multi-set 
$(3g_i-3+m_i)_{i=1}^k$, which corresponds to a partition of $d=\dim([G])$.
For a partition $\q\in \PP(d)$, let $\QQ(\q;g,n)$ denote the set of all 
multi-sets $\p=\{(g_i,m_i)\}_{i=1}^m$ of elements of 
 $Q$ such that $\q(\p)=\q$ and there is a \SCC $G$ with 
$$\p=\p_G,\ \ g=g(G),\ \ \text{and}\ \ n=n(G).$$
If $\q=(a_1\geq a_2\geq ...\geq a_k>0)$ is a partition of $d$ and 
$\p\in \QQ(\q;g,n)$,
after possible re-arrangement of the indices we have 
\begin{itemize}
\item $m\geq k,\ \ (g_i,m_i)=(0,3) \ \ \text{for}\ k<i\leq m$.
\item $0\leq g_i \leq \left\lfloor\frac{a_i+2}{3}\right\rfloor\ \ \text{for}\ 
1\leq i\leq k$.
\item $(k+d+2)-(2g+n)\leq \sum_{i=1}^k g_i \leq g$.
\end{itemize}
The last inequality follows since 
$$g=\left(\sum_{i=1}^kg_i\right)+|E(G)|-|V(G)|+1=
\left(\sum_{i=1}^kg_i\right)+(3g-3+n-d)-m+1,$$
and $m\geq k$.
\begin{lem}\label{lem:p-cycles}
For every $\q=(a_1\geq a_2\geq ...\geq a_k>0)\in P(d)$ and every multi-set $\p$ 
of $m$  elements $\{\theta_i=(g_i,m_i)\}_{i=1}^m$ in $Q$ the following are true.\\
\begin{enumerate}
\item The elements of 
$\QQ(\q;g,n)$ are in correspondence with the choice of genera $(g_1,...,g_k)$ 
satisfying the following two relations.
\begin{displaymath}%equation}\label{eq:conditions-on-p}
\begin{split}
&0\leq g_i \leq \left\lfloor\frac{a_i+2}{3}\right\rfloor\ \ \ \forall\ \ i=1,...,k,\\ 
&(d+k+2)-(2g+n)\leq \sum_{i=1}^k g_i \leq g .
\end{split}
\end{displaymath}%equation}
\item For every integer $e\geq m-1$ such that
\begin{displaymath}
\begin{split}
\sum_{i=1}^m m_i\geq 2e
\end{split}
\end{displaymath}
there is a corresponding connected \SCC $G=G(\p,e)$ with $e$ edges
such that $\p=\p_G$.
\end{enumerate}
\end{lem} 
\begin{proof}
We have already seen that for $\p\in \QQ(\q;g,n)$ 
the properties stated before Lemma~\ref{lem:p-cycles} are satisfied.
This immediately gives the genera $(g_1,...,g_k)$. On the other hand, 
if $(g_1,...,g_k)$ are given as above, one may set
$m_i:=a_i-3g_i+3$, for $i=1,...,k$, and $(g_i,m_i)=(0,3)$ for 
$k<i\leq m$, where 
$$m=2g-2+n+\sum_{i=1}^k(g_i-a_i)\geq k.$$
Consider the multi-set $\p=((g_i,m_i))_{i=1}^m$. The desired  
\SCC $G$ corresponding  to 
$\p$ should have $e=3g-3+n-d$ edges. Since
$$\sum_{i=1}^m m_i=3m+d-3\sum_{i=1}^kg_i=n+2e\geq 2e,$$
the second part of the lemma implies that there is a \SCC $G$ 
which corresponds to $\p$. It thus suffices to prove the second 
claim of the lemma.\\

 Suppose that a multi-set
$\p=(g_i,m_i)_{i=1}^m$ is given as above, and that 
$\sum_{i=1}^mm_i\geq 2e\geq (2m-2)$.
Assume that $m_1$ is the smallest of all $m_i$. If $m_1\geq 2$ let
$A$ be a set of $\sum_{i=1}^mm_i$ elements 
$$A=\{(i,j)|i=1,...,m, j=1,...,m_i\},$$
and $B$ denote a set of $e$ (disjoint) pairs of elements of $A$, where 
$m-1$ of the pairs are the following
$$((1,2),(2,1)),((2,2),(3,1)),...,((m-1,2),(m,1)).$$
This is possible since the total number of points in $A$ is at least $2e$, $e\geq m-1$,
and each $m_i$ is at least $2$. 
Let $G$ be the graph with vertices $V(G)=\{1,...,m\}$ and 
$E(G)$ be a set of $e$ edges, where for every pair
$((i,p_i),(j,p_j))\in B$ we draw an edge between $i$ and $j$ in $E(G)$.
Finally, let $\epsilon_G(i)=(g_i,m_i-d_i)$. It is straightforward to check that 
$G$, together with the weight function $\epsilon_G$ is a \SCC.\\

If $m_1=1$, let
$$\p'=(g_i,m_i)_{i=2}^m,\ \ \ e'=e-1.$$
Clearly, $e'\geq m-2$, and using an induction on the number of vertices,
there is a corresponding \SCC $G'$ with $e'$ edges and $\p'=\p_{G'}$.
Suppose that the vertex $i\in\{2,...,m\}$ has degree $d_i$ in $G'$ and
assume that $\epsilon_{G'}(i)=(g_i,n_i)$. Thus we have
$m_i=d_i+n_i$. Note that
\begin{equation}\label{eq:graph}
\begin{split}
&2e-2%\left\lceil\frac{m_1+1}{2}\right\rceil
=\left(\sum_{i=1}^mm_i\right)-
\left(\sum_{i=2}^mn_i\right)-1\geq 2e-\left(\sum_{i=2}^mn_i\right)-1\\
\Rightarrow\ \ \ &
\sum_{i=2}^m n_i\geq 1. %2\left\lceil\frac{m_1+1}{2}\right\rceil-m_1.
\end{split}
\end{equation}

Let $G$ be the \SCC obtained by adding a vertex
$1$ to $G'$, assigning the genus $g_1$ to $1$, and letting $n_1=0$. We attach
one 
%$\lfloor\frac{m_1-1}{2}\rfloor$  connecting $1$  $2\lceil\frac{m_1+1}{2}\rceil-m_1$   
edge connecting $1$ to one of the vertices
corresponding to the non-zero values of $n_i,\ i=2,...,m$.
The last inequality in (\ref{eq:graph}) implies that this
is always possible. 
The stable weighted graph $G$ will then have the required properties.
This completes the argument in the second case (i.e. $m_1=1$) by induction on the
number $m$ of vertices.\\
\end{proof}

\section{Preliminaries on combinatorial cycles}
\label{sec:argument}
Let us assume that the genus $g$, the number $n$ of the marked points and 
the degree $d$ are given as before.
 Set $e=3g-3+n-d$, and let $\QQ(d;g,n)$ denote the union of 
 $\QQ(\q;g,n)$ for $\q\in \PP(d)$, and $\PP(d;g,n)$ denote the subset of
$\PP(d)$ consisting of $\q\in \PP(d)$ which are of the form $\q(\p)$ for
some $\p\in \QQ(d;g,n)$.
%Denote the set of partitions of $d$ into at most
%$r$ summands by $\PP(d,r)$. 
%The number of partitions in $\PP(d,r)$  is denoted by $p(d,r)$.

\begin{prop}\label{prop:rank-of-combinatorial-kappa-ring}
For $d\leq n+3g-3$ we have $\PP(d;g,n)=\PP(d,3g-2+n-d)$.
\end{prop}
\begin{proof}
We  first show that $\PP(d;g,n)\subset\PP(d,3g-2+n-d)$.
Let $\q=\q(\p)$ for $\p=\p_G=(g_i,m_i)_{i=1}^m$ which corresponds to a 
\SCC $G$. We may assume that
\begin{displaymath}
\begin{split}
&a_i=3g_i-3+m_i,\ \ \ \ \text{for}\ \ 1\leq i\leq k,\\
&g_i=m_i-3=0,    \ \ \ \ \ \ \text{for}\ \ k<i\leq m.
\end{split}
\end{displaymath}
From here
\begin{displaymath}
\begin{split}
2g&=2\left(\sum_{i=1}^kg_i\right)+2|E(G)|-2|V(G)|+2\\
&=2\left(\sum_{i=1}^kg_i\right)+\left(\left(\sum_{i=1}^mm_i\right)-n\right)-2m+2.
\end{split}
\end{displaymath}
We thus have
\begin{displaymath}
\begin{split}
(i)\ \ \ \ &\sum_{i=1}^k(2g_i-2+m_i)=2g-2+n+(k-m)\leq 2g-2+n,\\
(ii)\ \ \ \ &\sum_{i=1}^k(2g_i-3+m_i)
=\left(\sum_{i=1}^k(3g_i-3+m_i)\right)-\left(\sum_{i=1}^kg_i\right)\geq d-g\\
%(3)\ \ \ \ &\sum_{i=1}^kg_i\leq g\\
&(i)-(ii)\ \Rightarrow\ \ k\leq 3g-2+n-d.
\end{split}
\end{displaymath}

For the reverse inclusion,  first suppose that $d\leq n+2g-2$, and that a partition
$\q=(a_1,...,a_k)$
of $d$ with $k\leq 3g-2+n-d$  is given. Define
\begin{displaymath}
\begin{split}
&r:=\min\{g,k\},\ \ \ m:=(2g-2+n-d)+\min\{g,k\}\\
&g_i:=\begin{cases}
1\ \ \ \ \ &\text{if }\ \ 1\leq i\leq r\\
0\ \ \ \ \ &\text{if }\ \ r< i\leq m
\end{cases},\ \ \& \ \ \
m_i:=\begin{cases}
a_i\ \ \ \ &\text{if }\ \ 1\leq i\leq r\\
a_i+3 &\text{if }\ \ r<i\leq k\\
3&\text{if }\ \ k<i\leq m
\end{cases},\\
& \p:=\left((g_i,m_i)\ \big|\ i\in\{1,...,m\}\right).
\end{split}
\end{displaymath}
The assumptions on $d$ and $k$ imply that $r\geq 0$, while $\q(\p)=\q$. Moreover, 
 \begin{displaymath}
 \sum_{i=1}^mm_i=d+3(m-r)=n+2(3g-3+n-d)\geq 2(3g-3+n-d).
 \end{displaymath}
 Then Lemma~\ref{lem:p-cycles} implies that $\p\in \QQ(d;g,n)$, and consequently,
$\q\in\PP(d;g,n)$.\\

Thus, it suffices to prove the inclusion for
$d=2g-2+n+r$, where $0<r<g$, i.e. to show that
$$\PP(d,g-r)\subset \PP(d;g,n).$$
Let  $\q=(a_1\geq ...\geq a_k>0)$
be an element of $\PP(d,g-r)$. We first claim that
$$\sum_{i=1}^k \left\lfloor\frac{a_i}{3}\right\rfloor\geq r.$$
If the above inequality is not satisfied 
\begin{displaymath}
\begin{split}
r-1&\geq \sum_{i=1}^k \left\lfloor\frac{a_i}{3}\right\rfloor \geq \sum_{i=1}^k\frac{a_i-2}{3}=\frac{d-2k}{3}\\
&\geq \frac{(2g-2+n+r)-2(g-r)}{3}=\frac{3r-2}{3}.
\end{split}
\end{displaymath}
This contradiction proves the claim. \\

Choose the
integers $0\leq \epsilon_i\leq \lfloor\frac{a_i}{3}\rfloor$, $i=1,...,k$
so that
$$\left(\sum_{i=1}^k \left\lfloor\frac{a_i}{3}\right\rfloor\right)
-\left(\sum_{i=1}^k \epsilon_i\right)=r.$$
Set
\begin{displaymath}
\begin{split}
&g_i:=\left\lfloor\frac{a_i}{3}\right\rfloor+1-\epsilon_i,\ \ \ \&\\
&m_i:=a_i-3\left\lfloor\frac{a_i}{3}\right\rfloor+3\epsilon_i,\\
\Rightarrow\ \ \ \ &3(g_i-1)+m_i=a_i,\ \ \ i=1,...,k.
\end{split}
\end{displaymath}
Note that $(g_i,m_i)\in Q$, and $\p=(g_i,m_i)_{i=1}^k$ is a multi-set with
$\q=\q(\p)$. For $e=(3g-3+n)-d=g-1-r$ we have
$n=n(\p,e)$ and $g=g(\p,e)$. Moreover, we have
\begin{displaymath}
\sum_{i=1}^{k}m_i=d-3r=2(g-1-r)+n\geq 2(g-1-r)=2e,
\end{displaymath}
and Lemma~\ref{lem:p-cycles} implies that $\p\in \QQ(d;g,n)$.
Thus $\q\in \PP(d;g,n)$, and the proof is complete.
\end{proof}

\section{The general strategy for obtaining lower bounds}
\label{sec:strategy}
Let us define a partial order on $\PP(d)$  by setting $\q_1\lhd \q_2$ if $\q_1$ refines
$\q_2$. Thus, $\langle \q,\p\rangle$ is non-zero only if $\q$ refines $\q(\p)$.
For $\p=((g_i,m_i))_{i=1}^m$  
\begin{displaymath}
\begin{split}
\Lambda(\p):=\langle \q(\p),\p\rangle=\prod_{i=1}^m\frac{1}{24^{g_i}\times g_i!}\neq 0
%&\langle \q(\p),\p\rangle=\prod_{i=1}^m\lambda_{g_i,m_i}\neq 0,\\
%\text{where }\ \ &\lambda_{g,n}:=\int_{[\Mgnbar]}\psi(3g-3+n).
\end{split}
\end{displaymath}
The rank $(r(d;g,n)$ of  the matrix
$$R(d;g,n):=\left(\frac{1}{\Lambda(\p)}\Big\langle \q,\p\Big\rangle\right)_{
\substack{\q\in \PP(d)\\ \p\in \QQ(d;g,n)}}$$
is a lower bound for the rank of $\kappa^d(\Mgnbar)$.
Let $\langle \kappa_1,...,\kappa_{3g-3+n}\rangle^d_\Q$ denote the
free $\Q$ module (formally) generated by the $\kappa$ classes in degree $d$.
The matrix $R(d;g,n)$ gives a surjective linear map
\begin{displaymath}%equation}\label{eq:Matrix}
R(d;g,n):\Big\langle \kappa_1,...,\kappa_{3g-3+n}\Big\rangle^d_\Q
\lra \Q^{r(d;g,n)},
\end{displaymath}%equation}
and  a surjection
\begin{displaymath}%equation}\label{eq:embedding}
\jmath_{g,n}^d:\kappa^d\left(\Mgnbar\right)\hookrightarrow
\frac{\Big\langle \kappa_1,...,\kappa_{3g-3+n}
\Big\rangle^d_\Q}{\Ker\left(R(d;g,n)\right)}.
\end{displaymath}%equation}
\begin{defn}
We call a class $\psi\in \Acal^*(\Mgnbar)$ {\emph{combinatorially trivial}}
if for every  \SCC $G$ 
$$\int_{[G]}\psi=0.$$
We denote the quotients of $\Acal^*(\Mgnbar)$ and $\kring$ by combinatorially trivial
tautological classes by $\Acal^*_c(\Mgnbar)$ and $\kringc$ respectively.
\end{defn}
We may summarize our considerations in the following proposition.
\begin{prop}\label{prop:main-1}
The  induced map
$$\jmath_{g,n}^d:\kappa^d_c\left(\Mgnbar\right)\ra
\frac{\Big\langle \kappa_1,...,\kappa_{3g-3+n}
\Big\rangle^d_\Q}{\Ker\left(R(d;g,n)\right)}$$
is an isomorphism.
\end{prop}
We would now like to describe a general strategy for achieving   
lower bounds for the rank of $\kappa^d(\Mgnbar)$. The strategy will be 
implemented in a few cases in the following sections.

\begin{defn}\label{def:fine-assignment}
Fix the genus $g$ and the number $n$ of the marked points for $\Mgnbar$,
as well as the degree $d$. Let $<$ denote a total ordering on 
$\PP(d,3g-2+n-d)$ which 
refines the partial order $\lhd$.
A \emph{fine assignment} (with respect to a subset $P$ of $\PP(d)$ 
and the total order $<$) is a function 
$f:P \ra \PP(d,3g-2+n-d)$ 
%which contains $P(d,3g-2+n-d)$, 
satisfying the following properties:\\
\begin{displaymath}%equation}\label{eq:fine-assignment}
\begin{split}
&(1)\ \q\lhd f(\q)
\\
&(2)\ \q\lhd \q',\  \ \q'\in \PP(d,3g-2+n-d)\ \ \Rightarrow\ \ 
f(\q)<\q'. 
\\
\end{split}
\end{displaymath}%equation} 
\end{defn}
Fix a fine assignment $f:P\ra \PP(d,3g-2+n-d)$. Consider a block decomposition of 
$R(d;g,n)$ where, for every $\q_0\in \PP(d,3g-2+n-d)$, the rows 
corresponding to all $\p\in \QQ(\q_0;g,n)$ belong to the same block, while the 
columns corresponding to all $\q\in f^{-1}(\q_0)$ belong to the same block as well.
In particular, associated with every such $\q_0$ we may introduce the matrix 
$$R_f(\q_0;g,n)=\left(
\frac{1}{\Lambda(\p)}\big\langle \q,\p\big\rangle
\right)_{\substack{\p\in \QQ(\q_0;g,n)\\ \q\in f^{-1}(\q_0)}},
$$
and will denote its rank by $r_f(\q_0;g,n)$. 
\begin{lem}\label{lem:rank-general}
Suppose that $f:P\ra \PP(d,3g-2+n-d)$ is a fine assignment as above. Then
$$r(d;g,n)\geq \sum_{\q\in \PP(d,3g-2+n-d)}r_f(\q;g,n).$$
\end{lem}
\begin{proof}
The fine assignment $f$ determines a block decomposition of a sub-matrix of 
$R(d;g,n)$ which is upper triangular with respect to the order $<$.
Since the matrices $R_f(\q;g,n)$ correspond to the diagonal in this block form,
the above lemma follows.
\end{proof}
\begin{remark}\label{remark:1}
More generally, let $I$ be a totally ordered set with the order $<$ and $f_1:\PP(d;g,n)\ra I$
and $f_2:\QQ(d;g,n)\ra I$ be surjective functions so that 
\begin{itemize}
\item For $\q,\q'\in\PP(d;g,n)$ with $\q\lhd \q'$, $f_1(\q)<f_1(\q')$.
\item For $\p,\p'\in \QQ(d;g,n)$ with $\q(\p)\lhd \q(\p')$, $f_2(\p)<f_2(\p')$.
\item For $\q\in \PP(d;g,n)$ and $\p\in\QQ(d;g,n)$ with $\q\lhd \q(\p)$, 
 $f_1(\q)<f_2(\q)$.
\end{itemize}
Then $I$ determines a block decomposition of $R(d;g,n)$, and $R(d;g,n)$ is upper 
triangular with respect to this decomposition. Thus 
$$r(d;g,n)\leq \sum_{p\in P}\mathrm{rank}\left(R_{f_1,f_2}(p)\right)$$
where $R_{f_1,f_2}(p)$ is the sub-matrix of $R(d;g,n)$ determined by the columns 
corresponding to $f_1^{-1}(p)$ and the rows corresponding to $f_2^{-1}(p)$.  
\end{remark}
The restriction of every fine assignment to 
$P\cap \PP(d,3g-2+n-d)$ is the identity.
Theorem~\ref{thm:main-2} is now 
an immediate consequence of Lemma~\ref{lem:rank-general}.
\begin{cor}
The rank of $\kappa_c^d(\Mgnbar)$ is greater than or equal to $p(d,3g-2+n-d)$.
\end{cor}
\begin{proof}
Take $P=\PP(d;g,n)=\PP(d,3g-2+n-d)$, $f:P\ra P$ the identity map, 
and  $<$  any refinement of $\lhd$. Since $r_f(\q;g,n)=1$ for all
$\q\in P$, we are done.
\end{proof}

\section{The combinatorial kappa ring in codimension one}
In this section, we apply Lemma~\ref{lem:rank-general} to the study 
of the rank of $\kappa_c^{3g-4+n}(\Mgnbar)$. We will first handle the case
$g=1$ using explicit formulas for the integrals of the $\psi$ classes.\\
 
\subsection{The combinatorial kappa ring of $\Modbar_{1,n}$ in codimension one}
Denote the set of $k$-element subset of $N=\{1,...,n\}$ by 
${N\choose k}$. For every $n$-tuple of real numbers 
$a_1,...,a_n$, denote the $i$-th symmetric product of them by 
$\sig_i(a_1,...,a_n)$. In other words, $\sig_0(a_1,...,a_n)=1$ and
\begin{displaymath}
\sig_i(a_1,...,a_n)=
\sum_{1\leq j_1<j_2<...<j_i\leq n}\left(\prod_{p=1}^ia_{j_p}\right).
\end{displaymath}
\begin{thm}\label{thm:symmetry-2}
Suppose that the non-negative integers $a_1,...,a_n$ are given so that 
$a_1+...+a_n=n$. Then
\begin{equation}\label{eq:psi-integrals}
\int_{\Modbar_{1,n}}\prod_{j=1}^n \psi_j^{a_j}=\frac{1}{24}{n\choose {a_1,...,a_n}}
\left(1-\sum_{i=2}^n\frac{\sig_i(a_1,...,a_n)}{i(i-1){n\choose i}}\right).
\end{equation}
\end{thm}
\begin{proof}
Suppose that $a_1\geq ...\geq a_m>0$ and $a_{m+1}=...=a_n=0$.
Set $M=\{1,...,m\}$ and $\q=(a_1,...,a_m)\in \PP(n)$.
If $m=n$, then $a_1=...=a_n=1$. In this case, 
\begin{displaymath}
\begin{split}
\frac{1}{24}{n\choose {1,...,1}}
\left(1-\sum_{i=2}^n\frac{\sig_i(1,...,n)}{i(i-1){n\choose i}}\right)
&=\frac{n!}{24} \left(1-\sum_{i=2}^n\frac{{n\choose i}}{i(i-1){n\choose i}}\right)\\
%&=\frac{n!}{24}\left(1-\sum_{i=2}^n\frac{1}{i(i-1)}\right)\\
%&=\frac{n!}{24}\left(1-\sum_{i=2}^n\left(\frac{1}{i-1}-\frac{1}{i}\right)\right)\\
&=\frac{(n-1)!}{24}=\int_{\Modbar_{1,n}}\prod_{j=1}^n \psi_j.
\end{split}
\end{displaymath}

We may thus assume that $n>m$. The String equation may be applied to the 
left-hand-side of (\ref{eq:psi-integrals}). Thus,
\begin{equation}\label{eq:psi-integral-string}
\int_{\Modbar_{1,n}}\prod_{j=1}^m \psi_j^{a_j}=\sum_{p=1}^m
\int_{\Modbar_{1,n}}\psi_p^{a_p-1}\prod_{j\in\{1,...,m\}- \{p\}}\psi_j^{a_j}.
\end{equation}

For $I=(i_1<...<i_k)\in {M\choose k}$ and
$$J=\{j_1<...<j_{l}\}\subset I,\ \ \ \ I-J=\{j_1^\circ<...<j_{k-l}^\circ\}$$ 
we define 
\begin{displaymath}\begin{split}
&\sig(I):=a_{i_1}+...+a_{i_k},\ \ \ \ \ |I|:=k\\
&{\sig(I) \choose \q(I)}:={\sig(I)\choose {a_{i_1},...,a_{i_k}}}\\
&{{\sig(I)-|J|}\choose {\q_J(I)}}:={{\sig(I)-|J|}\choose
 {a_{j_1}-1,...,a_{j_l}-1,a_{j_1^\circ},...,a_{j_{k-l}^\circ}}}.
\end{split}\end{displaymath}

 The right-hand-side of (\ref{eq:psi-integrals}) may be re-written using
 \begin{equation}
 \begin{split}
\sum_{i=2}^n\frac{\sig_i(a_1,...,a_n)}{i(i-1){n\choose i}}
&=\frac{1}{n!}\sum_{i=2}^m\left((i-2)!\sum_{I\in {M\choose i}}
{{n-i}\choose \q_I(M)}\right).
 \end{split}
 \end{equation}
 Thus, the right-hand-side satisfies the String equation as well, since 
 $${{b_1+...+b_m}\choose {b_1,...,b_m}}=\sum_{k=1}^m
 {{b_1+...+b_m-1}\choose {b_1,...,b_{k-1},b_k-1,b_{k+1},...,b_m}}.$$
 Together with the verification of (\ref{eq:psi-integrals}) for 
 $a_1=...=a_n=1$, this completes the proof by induction.
\end{proof}

Consider the graphs shown in Figure~\ref{fig:graphs} together with the 
illustrated weight functions, which determine stable weighted graphs 
$G_i$ $i=1,...,n$.  With the notation of the introduction 
$D_i=[G_i]$ for $i=1,...,n-1$, while  $D_n=2[G_n]$, since $|Aut(G_n)|=2$.

\begin{thm}\label{thm:divisor-relation}
For every element $\kappa\in\kappa^{n-1}(\Modbar_{1,n})$
\begin{equation}\label{eq:genus-one-relation}
\frac{1}{24}\int_{[G_n]}\kappa=\sum_{i=1}^{n-1}{{n-2}\choose i-1}\int_{[G_i]}\kappa.
\end{equation}
\end{thm}
\begin{proof}
It suffices to prove the theorem for all $\psi$ classes. Let $\tilde\q=(b_1,...,b_k)\in\PP(n-1)$
be a partition of $n-1$ of length $k$ and set $\q=(a_1,...,a_k)$ where 
$a_i=b_i+1$ for $i=1,...,k$. Set 
$$F(\q)=F(a_1,...,a_k)=\int_{[G_n]}\psi(\tilde{\q})-24\sum_{i=1}^{n-1}
{{n-2}\choose i-1}\int_{[G_i]}\psi(\tilde{\q}).$$
With $N=\{1,...,k\}$, $I^\circ=N-I$ for every $I\subset N$, 
and following the notation set in the proof of Theorem~\ref{thm:symmetry-2} 
\begin{displaymath}
\begin{split}
F(\q)&={{\sigma(N)}\choose \q(N)}-\sum_{I\subset N}{{\sigma(N)-|N|-1}\choose 
{\sigma(I)-|I|-1}}{{\sigma(I)}\choose{\q(I)}}{{\sigma(I^\circ)}\choose {\q(I^\circ)}}\\
&\ \ \ \ \ \ \ \ \ \ \ \ \ +\sum_{J\subset I\subset N}(|J|-2)!{{\sigma(N)-|N|-1}\choose 
{\sigma(I)-|I|-1}}{{\sigma_J(I)}\choose{\q_J(I)}}{{\sigma(I^\circ)}\choose {\q(I^\circ)}}.
\end{split}
\end{displaymath}
The above equation may be used to define the function $F$ for every 
partition $\q$ (relaxing the condition $a_i>1$ for $i=1,...,k$).
Set $$\q_i=(a_1,...,a_{i-1},a_i-1,a_{i+1},...,a_k).$$
Assuming $a_i>1$ for $i=1,...,k$, and setting 
$\lambda(I)={\sig(I)-|I|-1}$ 
\begin{displaymath}
\begin{split}
F(\q_i)=&\frac{a_i}{\sig(N)}{{\sig_i(N)}\choose{\q_i(N)}}-\sum_{i\in I\subset N}
\frac{a_i\lambda(I)}{\sig(I)\lambda(N)}{{\lambda(N)}\choose 
{\lambda(I)}}{{\sigma(I)}\choose{\q(I)}}{{\sigma(I^\circ)}\choose {\q(I^\circ)}}\\
& -\sum_{i\in I^\circ\subset N}
\frac{a_i(\lambda(I^\circ)+1)}{\sig(I^\circ)\lambda(N)}{{\lambda(N)}\choose 
{\lambda(I)}}{{\sigma(I)}\choose{\q(I)}}{{\sigma(I^\circ)}\choose {\q(I^\circ)}}\\
&+\sum_{\substack{J\subset I\subset N\\ i\in J}}
\frac{(a_i-1)\lambda(I)}{\sig_J(I)\lambda(N)}(|J|-2)!{{\lambda(N)}\choose 
{\lambda(I)}}{{\sigma_J(I)}\choose{\q_J(I)}}{{\sigma(I^\circ)}\choose {\q(I^\circ)}}\\
& +\sum_{\substack{J\subset I\subset N\\ i\in I-J}}
\frac{a_i\lambda(I)}{\sig_J(I)\lambda(N)}(|J|-2)!{{\lambda(N)}\choose 
{\lambda(I)}}{{\sigma_J(I)}\choose{\q_J(I)}}{{\sigma(I^\circ)}\choose {\q(I^\circ)}}\\
&+\sum_{\substack{J\subset I\subset N\\ i\in I^\circ}}
\frac{a_i(\lambda(I^\circ)+1)}{\sig(I^\circ)\lambda(N)}(|J|-2)!{{\lambda(N)}\choose 
{\lambda(I)}}{{\sigma_J(I)}\choose{\q_J(I)}}{{\sigma(I^\circ)}\choose {\q(I^\circ)}}.
\end{split}
\end{displaymath} 
Summing over $i=1,...,k$ we obtain
\begin{displaymath}
\begin{split}
\sum_{i=1}^kF(\q_i)=&{{\sig(N)}\choose{\q(N)}}-\sum_{I\subset N}
\frac{\lambda(I)+(\lambda(I^\circ)+1)}{\lambda(N)}{{\lambda(N)}\choose 
{\lambda(I)}}{{\sigma(I)}\choose{\q(I)}}{{\sigma(I^\circ)}\choose {\q(I^\circ)}}\\
&+\sum_{J\subset I\subset N}
\frac{\lambda(I)+(\lambda(I^\circ)+1)}{\lambda(N)}(|J|-2)!
{{\lambda(N)}\choose {\lambda(I)}}
{{\sigma_J(I)}\choose{\q_J(I)}}{{\sigma(I^\circ)}\choose {\q(I^\circ)}}.
\end{split}
\end{displaymath}
Thus
\begin{equation}\label{eq:combinatorial-string}
\sum_{i=1}^kF(\q_i)=F(\q).
\end{equation}
If  $\ov{\q}=(a_1,...,a_{k},1)$ is obtained by adding a $1$ to $\q$,
\begin{displaymath}
\begin{split}
F(\ov{\q})=&(\sig(N)+1){{\sig(N)}\choose{\q(N)}}
-\sum_{\substack{I\subset N\\ \ov{I}=I\cup\{k+1\}}}(\sig(I)+1)
{{\lambda(N)}\choose{\lambda(I)}}{{\sig(I)}\choose{\q(I)}}{{\sig(I^\circ)}\choose{\q(I^\circ)}}\\
&-\sum_{\substack{I\subset N\\ \ov{I}=I}}
(\sig(I^\circ)+1)
{{\lambda(N)}\choose{\lambda(I)}}{{\sig(I)}\choose{\q(I)}}{{\sig(I^\circ)}\choose{\q(I^\circ)}}\\
&+\sum_{\substack{J\subset I\subset N\\ \ov{J}=J\cup\{k+1\}, \ov{I}=I\cup\{k+1\}}}
\left((|J|-1)!{{\lambda(N)}\choose 
{\lambda(I)}}{{\sigma_J(I)}\choose{\q_J(I)}}{{\sigma(I^\circ)}\choose {\q(I^\circ)}}\right)\\
&+\sum_{\substack{J\subset I\subset N\\ \ov{J}=J, \ov{I}=I\cup\{k+1\}}}
(\sig_J(I)+1)\left((|J|-2)!{{\lambda(N)}\choose 
{\lambda(I)}}{{\sigma_J(I)}\choose{\q_J(I)}}{{\sigma(I^\circ)}\choose {\q(I^\circ)}}\right)\\
&+\sum_{\substack{J\subset I\subset N\\ \ov{J}=J, \ov{I}=I}}
(\sig(I^\circ)+1)\left((|J|-2)!{{\lambda(N)}\choose 
{\lambda(I)}}{{\sigma_J(I)}\choose{\q_J(I)}}{{\sigma(I^\circ)}\choose {\q(I^\circ)}}\right)\\
&+\sum_{\substack{J=\{i\}\subset I\subset N\\ \ov{J}=J\cup\{i,k+1\}, \ov{I}=I\cup\{k+1\}}}
\frac{a_i}{\sig(I)}{{\lambda(N)}\choose 
{\lambda(I)}}{{\sigma(I)}\choose{\q(I)}}{{\sigma(I^\circ)}\choose {\q(I^\circ)}}\\
=&(\sig(N)+1)F(\q)-\sum_{\substack{I\subset N}}
{{\lambda(N)}\choose{\lambda(I)}}{{\sig(I)}\choose{\q(I)}}{{\sig(I^\circ)}\choose{\q(I^\circ)}}\\
&+\sum_{\substack{J=\{i\}\subset I\subset N}}
\frac{a_i}{\sig(I)}{{\lambda(N)}\choose 
{\lambda(I)}}{{\sigma(I)}\choose{\q(I)}}{{\sigma(I^\circ)}\choose {\q(I^\circ)}}\\
&\Rightarrow\ \ \ \ \ \ \ \ F(\ov\q)=(\sig(N)+1)F(\q).
\end{split}
\end{displaymath}
The above computation, together with (\ref{eq:combinatorial-string}),
reduce the proof to the case where $k=1$ and $a_1=1$, which is straight forward.
\end{proof}
\begin{figure}
\def\svgwidth{13cm}\begin{center}
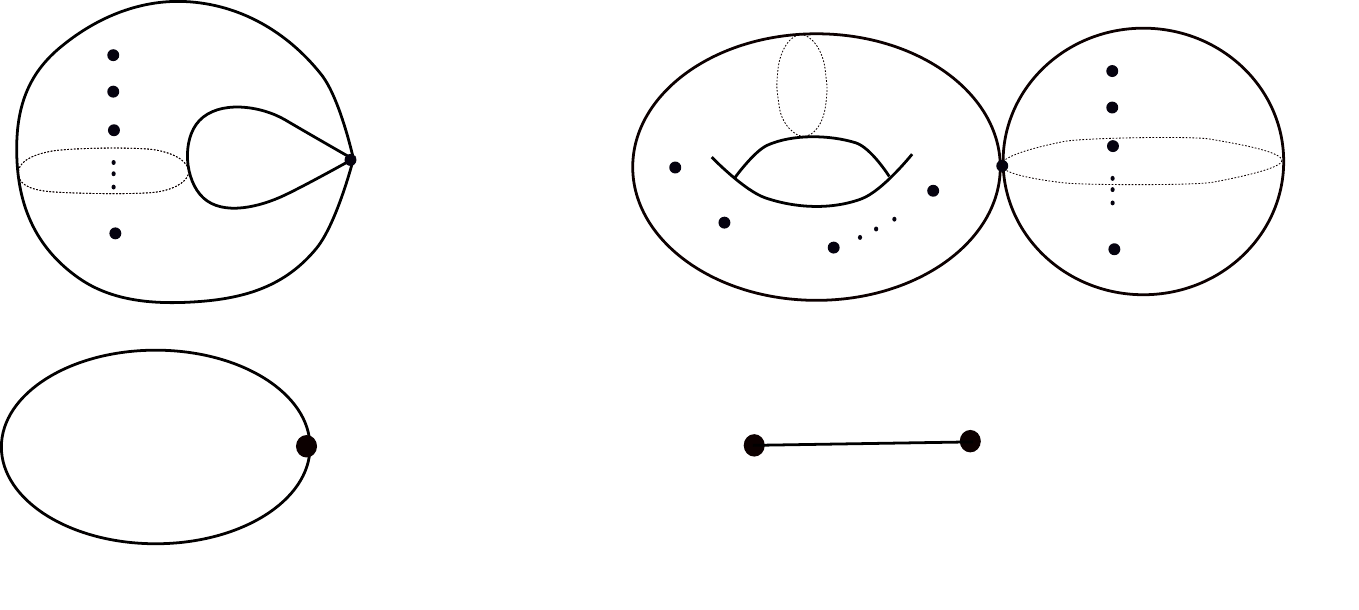
 \caption{\label{fig:graphs}\
{The \SCC $G_i$ for $i=1,...,n$ is illustrated. Each \SCC
corresponds to a divisor in $\Modbar_{1,n}$}}
\end{center}
\end{figure}
\begin{remark}
The above proof uses the combinatorial formulas for the terms appearing in $F(\q)$ 
to show that $F(\q)$ (in a sense) satisfies the String and the Dilation 
equations. One can present a purely geometric proof for these two equations, and 
obtain a proof of Theorem~\ref{thm:divisor-relation} which does not use 
Theorem~\ref{thm:symmetry-2}. 
\end{remark}
\begin{thm}\label{thm:genus-one-codim-one}
The rank of $\kappa_c^{n-1}(\Modbar_{1,n})$ is equal to $n-1$.
\end{thm}
\begin{proof}
Since the total number of rows in $R(n-1;1,n)$ is $n$, Theorem~\ref{thm:divisor-relation}
implies that $r(n-1;1,n)$ is at most $n-1$. In order to show the equality, let $P$ denote 
the union of $\PP(n-1,2)$ with the set of partitions $(a_1\leq a_2\leq a_3)$ of $n-1$
such that $a_1+a_2\leq a_3$. Define $f:P\ra \PP(n-1,2)$ by setting 
$f(a_1,a_2,a_3)=(a_1+a_2,a_3)$ and extending by the identity over $\PP(n-1,2)$. Thus,
$f$ defines a fine assignment with image in $\PP(n-1,2)$. For every 
$\q=(a_1\leq a_2)\in \PP(n-1;2)$ with $a_1>1$, 
$$\QQ(\q;1,n)=\left\{\left((1,a_1),(0,a_2+3)\right),\left((0,a_1+3),(1,a_2)\right)\right\}.$$
Both $\q$ and $\tilde{\q}=(1,a_1-1,a_2)$ correspond to the matrix $R_f(\q;1,n)$.
Thus, $r_f(\q;1,n)$ is greater than or equal to the rank of the $2\times 2$ matrix 
determined by $\QQ(\q;1,n)$, $\q$ and $\tilde{\q}$. Theorem~\ref{thm:symmetry-2}
may be used to compute the aforementioned matrix as 
\begin{displaymath}
M(\q)=\left(\begin{array}{cc}
1&24\int_{\Modbar_{1,a_1+2}}\psi_1^2\psi_2^{a_1}\\
1&{{a_1+2}\choose 2}
\end{array}\right)
=\left(\begin{array}{cc}
1&{{a_1+2}\choose 2}-a_1\\
1&{{a_1+2}\choose 2}
\end{array}\right).
\end{displaymath}
Thus, $r_f(\q;1,n)=2$ for every such $\q$. For $\q\in\{(n-1), (1,n-2)\}$ and for 
$\q=((n-1)/2,(n-1)/2)$ we have $r_f(\q;1,n)\geq 1$. These observations imply
that the rank of $R(n-1;1,n)$ is at least $n-1$.
\end{proof}

\subsection{The $\psi$ classes of length two}
In adapting a similar argument when $g>1$ we encounter matrices which 
play the role of the matrices $M(\q)$ encountered in the proof of 
Theorem~\ref{thm:genus-one-codim-one}. Dealing with these interesting matrices requires 
more work, as we will see in this subsection.\\

Fix $g>1$ and $m>3g-1$, and consider the matrix
\begin{displaymath}
\begin{split}
&M(m;g)=\left(n_j(h,m)\right)_{\substack{h=0,1,...,g\\ j=0,2,3,...,g+1}},\ \ \ 
\text{where}\\
&n_j(h,m):=24^h\times h!\times \int_{\Modbar_{h,m-3h+3}}\psi_1^{m-j}\psi_2^j
\ \ j=0,1,...,m.
\end{split}
\end{displaymath}
We begin our investigation with the study of the entries of this matrix. Let 
 $P_j(h)=n_j(h,3h-1)$.
\begin{lem}
%With the above notation fixed, 
$P_j$ is a polynomial of degree $j$, with the 
leading coefficient  $$\frac{6^j}{(2j+1)!!}.$$
\end{lem} 
\begin{proof}
The claim is trivial for $j=0$.
Note that by definition $$P_j(h)=\langle \tau_j\tau_{3h-1-j}\rangle.$$
Applying the KdV equation in the case $j\geq 1$ we get:
\begin{equation}\label{eq:KdV}
\begin{split}
P_j(h)=&\left(\frac{1}{2j+1}-2\right)P_{j-1}(h)+\left(\frac{1}{2j+1}-1\right)P_{j-2}(h)\\
&\ \ +\frac{1}{2j+1}\left({h \choose \frac{j}{3}}+2{h \choose \frac{j-1}{3}}\right)\\
&\ \ +\frac{6h}{2j+1} \left(\sum_{k=0}^4 {4\choose k}P_{j-1-k}(h-1)\right).
\end{split}
\end{equation}  
Both claims are then quick (inductive) implications of (\ref{eq:KdV}). 
\end{proof}
Let us assume that $P_j(h)=\sum_{k=0}^jA^j_{k}h^{j-k}$.
\begin{lem}\label{lem:String}
With the above notation fixed, $n_j(h,m)$ is a polynomial in the  variables $h$ and $m$
of degree $j$. If
$$n_j(h,m)=\sum_{\substack{p+q\leq i\\ p,q\geq 0}}A_j(p,q)h^pm^q,$$
then $A_j(p,q)\neq 0$ for $p+q=j$.
\end{lem}
\begin{proof}
For $j=0$, $n_{0}(h,m)=1$ and the claim is trivial. 
Suppose now that 
$$n_i(h,m)=\sum_{p+q\leq i}A_i(p,q)h^pm^q,$$
for $0\leq i<j$. Suppose that $m>3h-1$ is an arbitrary integer and $j>0$. 
From the String equation
\begin{displaymath}
\int_{\Modbar_{h,m-3h+3}}\psi_1^{m-j}\psi_2^j=
\int_{\Modbar_{h,m-3h+2}}\psi_1^{m-j-1}\psi_2^j
+\int_{\Modbar_{h,m-3h+2}}\psi_1^{m-j}\psi_2^{j-1}
\end{displaymath}
we obtain
\begin{equation}\label{eq:String}
\begin{split}
n_{j}(h,m)&=n_{j}(h,m-1)+n_{j-1}(h,m-1)\\
&=...=n_{j}(h,3h-1)+\sum_{k=3h-1}^{m-1}n_{j-1}(h,k)\\
&=P_j(h)+\sum_{p+q<j}A_{j-1}(p,q)h^p\left(\sum_{k=3h-1}^{m-1}k^q\right)
\end{split}
\end{equation}
Equation~(\ref{eq:String})  determines $n_j(h,m)$ as a polynomial 
of degree $j$. The degree $j$ part of $n_j$ may be computed from the degree 
$j-1$ part of $n_{j-1}$. More precisely, let $a_j(i):=A_j(j-i,i)$ and set 
$m_j(h,m)=\sum_{i=0}^ja_j(i)h^{j-i}m^i$. Then 
\begin{displaymath}
\begin{split}
\sum_{i=0}^ja_j(i)&h^{j-i}m^i=A^j_0h^j+\sum_{i=1}^{j}\frac{a_{j-1}(i-1) 
(h^{j-i}m^i-3^ih^j)}{i!}\\
&=\left(\frac{6^{j}}{(2j+1)!!}-\sum_{i=0}^{j-1}
\frac{a_{j-1}(i)3^{i+1}}{(i+1)!}\right)h^j
%\\ &\ \ \ \ \ \ \ \ 
+\sum_{i=1}^{j}\frac{a_{j-1}(i-1)}{i!}h^{j-i}m^i.
\end{split}
\end{displaymath}
From here 
\begin{displaymath}
\begin{split}
a_j(i)=\frac{a_{j-i}}{\prod_{k=1}^ik!},\ \ \text{where }
a_j=\begin{cases}
\frac{6^{j}}{(2j+1)!!}-\sum_{i=1}^j\frac{3^i}{\prod_{k=1}^ik!}a_{j-i}\ \ 
&\text{if }j>0\\
1 &\text{if }j=0\end{cases}.
\end{split}
\end{displaymath}
For every rational number $x$ and every prime number $p$ let
$\mathrm{ord}_p(x)$ denote the integer 
$k$ such that there are integers $a,b$  such that 
$x=p^k(a/b)$ and $p\nmid ab$. 
Using the above recursive formula and by induction on $j$
\begin{displaymath}
\begin{split}
&\mathrm{ord}_2(a_j)=-\sum_{i=1}^j\mathrm{ord}_2(i!)\ \ \ \forall\ \ j=0,1,2,\hdots\\
\Rightarrow\ \ &0=\mathrm{ord}_2(a_0)=\mathrm{ord}_2(a_1)>
\mathrm{ord}_2(a_2)>\mathrm{ord}_2(a_3)>\mathrm{ord}_2(a_4)>\hdots .
\end{split}
\end{displaymath}
Thus$a_j\neq 0$, and consequently $A_j(p,q)\neq 0$ for $p+q=j$.
\end{proof}
Lemma~\ref{lem:String} implies that we may naturally extend $n_j(h,m)$ and 
define it for the values of $h,m$ which do not necessarily satisfy $m\geq 3h-1$ 
or $m\geq j$. For $m\geq 3h-1$  we trivially have 
$n_j(h,m)=n_{m-j}(h,m)$. 
It happens that the aforementioned symmetry extends
to a slightly larger range of values.\\

\begin{lem}\label{lem:symmetry}
For every $j$ satisfying $j\geq 2h+1$, $n_j(h,j-1)=0$. 
\end{lem}
\begin{proof}
If $j\geq 3h$, by the String equation
$$n_j(h,j-1)=n_j(h,j)-n_{j-1}(h,j-1)=1-1=0.$$
The KdV equation implies that for every $i,j\neq 0$ 
\begin{equation}\label{eq:KdV-general}
\begin{split}
(2i+1)\langle \tau_i\tau_j\tau_0^{n+2}\rangle_h
=&(2n+1)\langle \tau_{i-1}\tau_j \tau_0^{n+1}\rangle_h+\frac{1}{4}\langle\tau_{i-1}\tau_j\tau_0^{n+4}\rangle_{h-1}\\
&+\frac{1}{24^h\times h!}\sum_{p=0}^{n+1} {{n+1}\choose{p}}
{g\choose{\frac{i-p}{3}}}
\end{split}
\end{equation}
Setting $i=m-j$ and $n=m-3h-1$, (\ref{eq:KdV-general}) implies
\begin{equation}\label{eq:KdV-replacement}
\begin{split}
n_j(h,m)&=\frac{2m-6h-1}{2m-2j+1}n_{j-1}(h,m-1)\\&\ +
\frac{6h}{2m-2j+1}n_{j-1}(h-1,m-1)+\frac{1}{2m-2j+1}q_j(h,m),\\
q_j(h,m)&:=\sum_{p=0}^{m-3h}{{m-3h}\choose p}{ h\choose {\frac{j-p}{3}}}
=\sum_{p=0}^{j}{{m-3h}\choose p}{ h\choose {\frac{j-p}{3}}}.
\end{split}
\end{equation}
The left-hand-side and the right-hand-side of the first equation 
in (\ref{eq:KdV-replacement}) are rational functions in $h,m$, and 
(\ref{eq:KdV-replacement}) is  
thus satisfied for all values of $h$ and $m$. Setting $m=j-1$ in 
(\ref{eq:KdV-replacement})
\begin{displaymath}
n_j(h,j-1)=(6h+3-j)n_{j-1}(h,j-2)-6hn_{j-1}(h-1,j-2)-q_j(h,j-1).
\end{displaymath}

Fixing $j$, all the terms in  $q_j(h,j-1)$ are zero $j\geq 3h+1$. If $j<3h+1$ 
$$q_j(h,j-1)=\sum_{k=0}^{\left\lfloor\frac{j}{3}\right\rfloor}
(-1)^{j-3k}{{3h-3k}\choose {j-3k}}{{h}\choose {k}}$$
is the coefficient of $x^j$ in 
\begin{displaymath}
\begin{split}
\left(\sum_{k=0}^h(-1)^k{h\choose k}x^{3k}\right)
\left(\sum_{\ell=0}^\infty {{3h-j-\ell}\choose \ell}x^{\ell}\right)&
=(1-x^3)^h\frac{1}{(1-x)^{3h-j+1}}\\
=(1+x+x^2)^h&(1-x)^{j-1-2h}.
\end{split}
\end{displaymath} 
For $2h+1\leq j <3h+1$ the above expression is a polynomial in $x$ of degree $j-1$ and 
the coefficient of $x^j$ in it is thus zero. This implies that $q_j(h,j-1)=0$ for all
$h$ satisfying $j\geq 2h+1$. Thus, for every $h,j$ satisfying  $j\geq 2h+1$ 
\begin{equation}\label{eq:n_j(h,j-1)=0}
n_j(h,j-1)=(6h+3-j)n_{j-1}(h,j-2)-6hn_{j-1}(h-1,j-2).
\end{equation}

Note that $n_j(0,m)={m\choose j}$, and $n_j(0,j-1)$ is thus zero. Let 
$h$ be the smallest genus such that there is some
$j$ with $j\geq 2h+1$ and $n_j(h,j-1)\neq 0$, and let $j$ be the largest such $j$. 
Then,  $j>2h$ and $j-1>2(h-1)$, implying $n_{j+1}(h,j)=n_j(h-1,j-1)=0$ by 
the minimality assumption on $(h,j)$. Then (\ref{eq:n_j(h,j-1)=0})
 gives $n_j(h,j-1)=0$. This contradiction proves the lemma.
\end{proof}
From Lemma~\ref{lem:symmetry}, for every $j\geq 2h$ 
\begin{displaymath}\begin{split}
n_j(h,j)&=n_{j+1}(h,j+1)-n_{j+1}(h,j)\\
&=n_{j+2}(h,j+2)-n_{j+2}(h,j+1)\ \ \ \left(\text{since }n_{j+1}(h,j)=0\right)\\
&=...=n_{j+3h}(h,j+3h)=1.
 \end{split}\end{displaymath}
 Since $n_0(h,m)=1$, this gives the equality 
 $$n_0(h,m)=n_m(h,m)\ \ \ \forall \ m\geq 2h.$$
\begin{prop}\label{prop:symmetry-1}
For every $m\geq 2h\geq 0$ and every $0\leq j\leq m$
$$n_j(h,m)=n_{m-j}(h,m).$$ 
\end{prop}
\begin{proof}
Denote the claim of the proposition for $(j,h,m)$ by 
$\Pcal(j,h,m)$, i.e. we claim that for $j=0,...,m$ and 
$m\geq 2h\geq 0$, $\Pcal(j,h,m)$ is true.\\

For $h=0$ and $0\leq j\leq m$, $\Pcal(0,j,m)$ is trivial.
Suppose that $h$ is the smallest genus such that
for some $0\leq j\leq m$ satisfying $m\geq 2h$, $\Pcal(h,j,m)$ is not true.
Take $m$ to be the largest possible value such that there is some 
$j$ with $\Pcal(j,h,m)$ true.
Fix $h,m$ as above and let $j$ be the largest integer with 
$\Pcal(j,h,m)$ true. Lemma~\ref{lem:symmetry}
implies that $j<m$. Moreover, the assumptions on $(j,h,m)$ implies 
that $\Pcal(j+1,h,m+1)$ and $\Pcal(j+1,h,m)$ are true.  The String equation 
$$n_j(h,m)=n_{j+1}(h,m+1)-n_{j+1}(h,m)$$
gives
$$\Pcal(j+1,h,m)\  \ \&\ \ \Pcal(j+1,h,m+1)\ \Rightarrow\ 
\Pcal(j,h,m).$$
This completes the proof of the proposition.
\end{proof}

Since in the column of $M(m;g)$ 
indexed by $j$ the coefficient of $h^j$ is a non-zero constant, 
 subtracting  appropriate multiples of the columns corresponding 
to $i=0,2,...,j-1$ from the column corresponding to $j$, for $j=2,...,g+1$
kills the monomials of degree $2,...,j-1$, while leaving the determinant unchanged.
The determinant 
$$d_g(m)=\mathrm{Det}(M(m;g))$$
is thus equal to the determinant of a matrix of the form
\begin{displaymath}
M'(m;g):=\Big(a_jh^j+b_{j-1}(m)h\Big)_{\substack{h=0,...,g\\ j=0,2,3,...,g+1}},
\end{displaymath}
where the constants $a_j$ are determined in the proof of Lemma~\ref{lem:String}
and $b_j(m)$ is a polynomial of degree at most $i$ in $m$ (and with $b_{-1}(m)=0$)
for $j=1,...,g$.\\

Considering the order of the coefficient $c_j$ of $m^j$ in $b_j(m)$ in the above 
process, one can easily observe that 
$$\mathrm{ord}_2(c_j)=-\sum_{i=1}^j\mathrm{ord}_2(i!),$$
and $c_j$ is thus always non-zero. As a consequence, $b_j(m)$ is a polynomial 
of degree $j$. \\

Subtracting $h$ times the row corresponding to $1$ from the row corresponding to 
$h$  for $h=2,...,g$ keeps the determinant unchanged. We thus have 
\begin{displaymath}
\begin{split}
d_g(m)&=\mathrm{Det}\left(
\begin{array}{cccc}
a_2+b_1(m)&a_3+b_2(m)&...&a_{g+1}+b_g(m)\\
(2^2-2)a_2&(2^3-2)a_3&...&(2^{g+1}-2)a_{g+1}\\
(3^2-3)a_2&(3^3-3)a_3&...&(3^{g+1}-3)a_{g+1}\\
\vdots &\vdots  &\ddots &\vdots\\
(g^2-g)a_2&(g^3-g)a_3&...&(g^{g+1}-g)a_{g+1}\\
\end{array}
\right)
\end{split}
\end{displaymath}
This implies that 
\begin{displaymath}
\begin{split}
d_g(m)&=g!\left(\prod_{h=2}^{g+1}a_h\right)\mathrm{Det}\left(
\begin{array}{ccccc}
1 &-\widehat{b}_1(m)&-\widehat{b}_2(m)&...&-\widehat{b}_g(m)\\
1&1&1&...&1\\
1&2&2^2&...&2^g\\
\vdots &\vdots &\vdots  &\ddots &\vdots\\
1&g&g^2&...&g^g
\end{array}
\right)\\
&=(-1)^{g+1}\left(\prod_{h=2}^{g}(a_hh!)\right)\left(
c_gm^g+\text{lower degree terms}
\right),
\end{split}
\end{displaymath}
where $\widehat{b}_j(m)=b_j(m)/a_{j+1}$. We conclude
\begin{lem}\label{lem-determinant}
The determinant $d_g(m)$ of the matrix $M(m;g)$ is a polynomial 
in $m$ of degree $g$.
\end{lem}
In other words, except for at most $g$ values of $m$,
the matrix $M(m;g)$ is a matrix of full rank. 

\begin{thm}\label{thm:determinant}
The determinant $d_g(m)$ is a (non-zero) constant multiple of 
$$(m-2)\prod_{i=g+3}^{2g+1}(m-i).$$
\end{thm}
\begin{proof}
By Lemma~\ref{lem-determinant}, it suffices to show that 
$$d_g(2)=d_g(g+3)=d_g(g+4)=...=d_g(2g+1)=0.$$
The first row of $M(2;g)$ consists of the following numbers
$$(1,{2\choose 2},{2\choose 3},...,{2\choose {g+1}})=(1,1,0,...,0).$$
Theorem~\ref{thm:symmetry-2} implies that the second row of 
$M(2;g)$ consists of the following numbers
$$(1,n_2(1,2),n_3(1,2),...,n_{g+1}(1,2))=(1,{2\choose 2}-{0\choose 1},0,...,0).$$
Thus the first two rows of $M(2;g)$ are equal and $d_g(2)=0$ for all $g\geq 1$.\\

For $g+3\leq m\leq 2g+1$, let $A(m;g)$ denote the sub-matrix of $M(m;g)$ 
which corresponds to the rows $h=0,...,\left\lfloor\frac{m}{2}\right\rfloor$
and the columns
$j=m-g-1,m-g,...,\left\lfloor\frac{m-1}{2}\right\rfloor$. Similarly, let 
$B(m;g)$ denote the sub-matrix of $M(m;g)$ 
which corresponds to the rows $h=0,...,\left\lfloor\frac{m}{2}\right\rfloor$
and the columns 
$j=\left\lceil\frac{m+1}{2}\right\rceil,\left\lceil\frac{m+1}{2}\right\rceil+1,...,g+1$.\\

Proposition~\ref{prop:symmetry-1} implies that the columns of $B(m;g)$
are the same as the columns of $A(m;g)$. Subtraction the $j$th column of 
$M(gm;g)$ from its $(m-j)$ column for 
$j=m-g-1,m-g,...,\left\lfloor\frac{m-1}{2}\right\rfloor$
produces a matrix with the same determinant, and with zeros in the block 
replaced for $B(m;g)$. If the determinant is non-zero the sum of the 
number of columns and the number of rows in $B(m;g)$ is at most $g+1$,
i.e. the total number of rows in $M(m;g)$. Thus, 
\begin{displaymath}
\begin{split}
&g+1\geq \left(\left\lfloor\frac{m}{2}\right\rfloor +1 \right)
+\left(g+2-\left\lceil\frac{m+1}{2}\right\rceil
\right)\\
\Leftrightarrow\ \ &\left\lceil\frac{m+1}{2}\right\rceil\geq 
\left\lfloor\frac{m}{2}\right\rfloor +2 .
\end{split}
\end{displaymath}
This contradiction implies that $d_g(m)=0$ for $m=g+3,g+4,...,2g+1$.
Since $d_g$ has at most $g$ roots, it is a constant multiple of 
$$(m-2)\prod_{i=g+3}^{2g+1}(m-i).$$
\end{proof}
\begin{remark}
 The final computation follows the law
\begin{displaymath}%equation}\label{eq:d_g}
d_g(m)=\frac{(-3)^{{g+1}\choose 2}}{(2g+1)!!\times g!}
(m-2)\prod_{i=g+3}^{2g+1}(m-i)
\end{displaymath}%equation}
for $g\leq 50$ according to our computations.
\end{remark}

\subsection{The rank of $\kappa_c^{3g-4+n}(\Mgnbar)$ for $g>1$}
As another application of Lemma~\ref{lem:rank-general} and Theorem~\ref{thm:determinant}
we examine the kappa ring in codimension one for $g>1$. 
With $d=3g-4+n$, we get $3g-2+n-d=2$. 
Through this section we will assume 
that $g>1$. We choose 
$P$ to be the union of $\PP(d;2)$ of partition of length $2$ with the 
set of all partitions of the form $\q'=(a_1> b_1\geq b_2>0)$ with 
$b_1+b_2\leq a_1$, and define $f(\q')=(a_1\geq a_2=b_1+b_2>0)$.
The order $<$ is defined by setting $(d)$ to be the largest element in $\PP(d,2)$,
and setting $(a_1\geq a_2>0)$ greater than $(b_1\geq b_2>0)$ if $a_1\geq b_1$.
Let us fix $\q=(a_1\geq a_2>0)$
and study the sub-matrices  $R_f(\q;g,n)$.\\

Let us first assume that $a_1>a_2\geq \max\{2,3g-2\}$, while 
$a_2\leq n-2$. For a corresponding partition 
$\p=\{(g_1,m_1),(g_2,m_2)\}\in \QQ(\q;g,n)$, $g_2$ is an arbitrary genus between 
$0$ and $g$, $g_1=g-g_2$, and $m_i=a_i+3-3g_i$, i.e. the matrix $R_f(\q;g,n)$ 
has precisely $g+1$ rows. The columns correspond to 
\begin{displaymath}
\q_0=(a_1\geq a_2),\ \ \ \q_j=(a_1\geq (a_2-j)\geq j)\ \ \ \ \ j=1,...,
\left\lfloor\frac{a_2}{2}\right\rfloor\geq \left\lfloor\frac{3g}{2}\right\rfloor-1.
\end{displaymath} 
The  $(g+1)\times (\lfloor a_2/2\rfloor+1)$  matrix
 $R_f(\q;g,n)$  consists of the entries $m_{hj}$, with $h=0,...,g$ and 
 $j=0,...,\lfloor a_2/2\rfloor$, 
 which are given by 
\begin{displaymath}
m_{hj}=\begin{cases}
\left(\int_{\overline{\Mod}_{g-h,a+3h}}\psi_1^{a_1+1}\right)
\left(\int_{\overline{\Mod}_{h,a_2-3h+4}}\psi_1^{a_2+1}\right)\ \ \ &\text{if }j=0\\
&\\
\left(\int_{\overline{\Mod}_{g-h,a+3h}}\psi_1^{a_1+1}\right)
\left(\int_{\overline{\Mod}_{h,a_2-3h+5}}\psi_1^{a_2-j+1}\psi_2^{j+1}\right)\ \ \ 
&\text{if }j\neq 0,\\
\end{cases}
\end{displaymath}
where $a=a_1-3g+4$. Thus, the rank of $R_f(\q;g,n)$ is equal to the rank of the matrix
$N(a_2+2;g)$, where 
\begin{displaymath}
\begin{split}
N(m;g)&:=\Big(n_{j}(h,m)\Big)_{\substack{h=0,...,g\\ j=0,2,3,...,\lfloor m/2\rfloor}}.
\end{split}
\end{displaymath}

Since $m=a_2+2\geq \max\{3g,4\}$, the matrix $M(g;m)$ is a sub-matrix of 
$N(g;m)$. This sub-matrix is full-rank by Theorem~\ref{thm:determinant}, and the
rank $r_f(\q;g,n)$ of $R_f(\q;g,n)$  is equal to $g+1$. Thus,
\begin{displaymath}
\sum_{\substack{\q=(a_1>a_2)\in P(d)\\ \max\{2,3g-2\}\leq a_2\leq n-2}}
r_f(\q;g,n)\geq \left(\sum_{\substack{\q=(a_1>a_2)\in P(d)\\ \max\{2,3g-2\}
\leq a_2\leq n-2}}(g+1)\right)
\end{displaymath}

When $a_1=a_2\geq 3g-2$, the possible combinatorial cycles correspond to the values 
$0\leq g_2\leq \lfloor g/2\rfloor$, and with a similar argument we have 
\begin{displaymath}
r_f(\q;g,n)=
\left\lfloor \frac{g+2}{2}\right\rfloor.
\end{displaymath}

For arbitrary values of  $a_2>4$, 
for $$\p=\big\{(g_1,m_1),(g_2,m_2)\big\}\in Q(\q;g,n)$$
we have 
$$\max\left\{0,\left\lceil \frac{a_2+2-n}{3}\right\rceil\right\}\leq g_2 
\leq \min\left\{\left\lfloor \frac{a_2+2}{3}\right\rfloor,g\right\}.$$ 
Let $\mathrm{row}_f(\q;g,n)$ denote  the number of rows 
in $R_f(\q;g,n)$. The matrix $R_f(\q;g,n)$ consists of the multiples 
of a subset of the rows in the matrix $N(a_2+2;h)$ where 
$h=\min\{g,\lfloor (a_2+2)/3\}$. Its rank is thus  equal to 
$\mathrm{row}_f(\q;g,n)$.\\

Finally, we gather the rows and the columns corresponding to the partitions
of the form $(a_1\geq a_2>0)$ with $a_2\in\{0,1,2,3,4\}$ in one block (see 
Remark~\ref{remark:1}). The 
corresponding partitions of $d$ consist of the following list:
$$A=\big\{(d),(d-1,1),(d-2,2),(d-3,3),(d-4,4)\big\}.$$
For every $\p\in \QQ(d;g,n)$ and every partition $\q$ in 
\begin{displaymath}
\begin{split}
\Big\{(d),(d-&1,1),(d-2,2),(d-2,1,1),(d-3,3),(d-3,2,1),(d-3,1,1,1),\\
&(d-4,4),(d-4,3,1),(d-4,2,2),(d-4,1,1,1,1)\Big\}
\end{split}
\end{displaymath}
if $\langle \q,\p\rangle \neq 0$ then  $\q(\p)\in A$.  
There are $11$ such $\p\in \QQ(d;g,n)$ which, together with the above $11$ 
partitions of $d$, determine  an $11\times 11$ sub-matrix
of $R_f(d;g,n)$. \\

Using the explicit formulas in Table~\ref{table:psi-integrals} for 
some of the $\psi$ integrals, one may compute the determinant of the above $11\times 11$
matrix. Surprisingly, the determinant is independent of $d$ and equals
\begin{displaymath}
-\frac{(g-1)^2(4928g^4-275516g^3-437138g^2+62924g-334941)}
{12936 \times 10^7}.
\end{displaymath} 
Thus, the aforementioned $11\times 11$ matrix is always of rank $11$.

\begin{table}
\caption{Small $\psi$ integrals} \label{table:psi-integrals}
\centering 
\begin{tabular}{|l|*{2}{c|}}\hline                
$\begin{array}{c}\\ \q=(a_1,...,a_k)\\ \\ \end{array}$&{The integral $\frac{1}{g!\times 24^g} 
\int_{\Modbar_{g,d(\q)+3-3g}}\prod_{i=1}^k\psi_i^{a_i}$} \\ \hline \hline           
 %%%%%%%%%%%%%%%%%%%%%%%%%%%%%%%%%%%%%%%%%%%%%%%%%%%%%%%%%%%
 $(d)$ &$\begin{array}{c}\\ 1\\  \end{array}$\\ \hline
 %%%%%%%%%%%%%%%%%%%%%%%%%%%%%%%%%%%%%%%%%%%%%%%%%%%%%%%%%%%
$( d,2)$&$\begin{array}{c}\\ {{d+2-g}\choose 2}+\frac{g(2g+3)}{5}\\  \end{array}$\\ \hline
 %%%%%%%%%%%%%%%%%%%%%%%%%%%%%%%%%%%%%%%%%%%%%%%%%%%%%%%%%%%
$( d,3)$&$\begin{array}{c}\\ {{d+3-g}\choose 3}+{{d+3-g}\choose 1}\frac{g(2g+3)}{5}-
\frac{g(8g^2+60g+37)}{105}\\ \end{array}$\\ \hline
 %%%%%%%%%%%%%%%%%%%%%%%%%%%%%%%%%%%%%%%%%%%%%%%%%%%%%%%%%%%
$( d,4)$&$\begin{array}{c}\\
{{d+4-g}\choose 4}+{{d+4-g}\choose 2}\frac{g(2g+3)}{5}
-{{d+4-g}\choose 1}\frac{g(8g^2+60g+37)}{105}\\
+\frac{g(g+1)(2g+3)(2g+5)}{70}\\ \end{array}$\\ \hline
 %%%%%%%%%%%%%%%%%%%%%%%%%%%%%%%%%%%%%%%%%%%%%%%%%%%%%%%%%%%
$( d,5)$&$\begin{array}{c}\\
{{d+5-g}\choose 5}+{{d+5-g}\choose 3}\frac{g(2g+3)}{5}
-{{d+5-g}\choose 2}\frac{g(8g^2+60g+37)}{105}\\
+{{d+5-g}\choose 1}\frac{g(g+1)(2g+3)(2g+5)}{70}
-\frac{g(2g+3)(8g^3+84g^2+55g+84)}{1155}\\ \end{array}$\\ \hline
 %%%%%%%%%%%%%%%%%%%%%%%%%%%%%%%%%%%%%%%%%%%%%%%%%%%%%%%%%%%
$( d,2,2 )$&$\begin{array}{c}\\
6{{d+4-g}\choose 4}+{{d+4-g}\choose 2}\frac{2g(2g+3)}{5}
+\frac{g(4g^3-4g^2-41g-9)}{25}\\
\end{array}$\\ \hline
 %%%%%%%%%%%%%%%%%%%%%%%%%%%%%%%%%%%%%%%%%%%%%%%%%%%%%%%%%%%
$( d,2,2,2)$&$\begin{array}{c}\\
90{{d+6-g}\choose 6}+{{d+6-g}\choose 4}\frac{18g(2g+3)}{5}
+{{d+6-g}\choose 2}\frac{3g(4g^3-4g^2-41g-9)}{25}\\
+\frac{g(8g^5-60g^4-70g^3+1275g^2+1067g+30)}{125}\\
\end{array}$\\ \hline
 %%%%%%%%%%%%%%%%%%%%%%%%%%%%%%%%%%%%%%%%%%%%%%%%%%%%%%%%%%%
$( d,2,2,2,2 )$&$\begin{array}{c}\\
2520{{d+8-g}\choose 8}+{{d+8-g}\choose 6}\frac{360g(2g+3)}{5}
+{{d+8-g}\choose 4}\frac{36g(4g^3-4g^2-41g-9)}{25}\\
+{{d+8-g}\choose 2}\frac{4g(8g^5-60g^4-70g^3+1275g^2+1067g+30)}{125}\\
+\frac{g(16 g^7 - 288 g^6+  1192 g^5+ 7440 g^4- 57671 g^3 - 120522 g^2- 34677 g
-20490)}{625}\\
\end{array}$\\ \hline
 %%%%%%%%%%%%%%%%%%%%%%%%%%%%%%%%%%%%%%%%%%%%%%%%%%%%%%%%%%%
( d,3,2)&$\begin{array}{c}\\
10{{d+5-g}\choose 5}+{{d+5-g}\choose 3}\frac{4g(2g+3)}{5}
-{{d+5-g}\choose 2}\frac{g(8g^2+60g+37)}{105}\\
+{{d+5-g}\choose 1}\frac{g(4g^3-4g^2-41g-9)}{25}
-\frac{g(2g+3)(8g^3+12g^2-467g-78)}{525}\\
\end{array}$\\ \hline
 %%%%%%%%%%%%%%%%%%%%%%%%%%%%%%%%%%%%%%%%%%%%%%%%%%%%%%%%%%%
$( d,4,2)$&$\begin{array}{c}\\
15{{d+6-g}\choose 6}+{{d+6-g}\choose 4}\frac{7g(2g+3)}{5}
-{{d+6-g}\choose 3}\frac{3g(8g^2+60g+37)}{105}\\
+{{d+6-g}\choose 2}\frac{g(76g^3+44g^2-419g-51)}{350}
-{{d+6-g}\choose 1}\frac{g(2g+3)(8g^3+12g^2-467g-78)}{525}\\
+\frac{g(g+2)(2g+1)(2g+3)(2g^2-11g-61)}{350}\\
\end{array}$\\ \hline
 %%%%%%%%%%%%%%%%%%%%%%%%%%%%%%%%%%%%%%%%%%%%%%%%%%%%%%%%%%%
$( d,3,3)$&$\begin{array}{c}\\
20{{d+6-g}\choose 6}+{{d+6-g}\choose 4}\frac{8g(2g+3)}{5}
-{{d+6-g}\choose 3}\frac{2g(8g^2+60g+37)}{105}\\
+{{d+6-g}\choose 2}\frac{2g(4g^3-4g^2-41g-9)}{25}
-{{d+6-g}\choose 1}\frac{2g(2g+3)(8g^3+12g^2-467g-78)}{525}\\
+\frac{g(64g^5+384g^4-13376g^3-76224g^2-71315g-15933)}{11025}\\
\end{array}$\\ \hline
\end{tabular}
\label{tab:hresult}
\end{table}

Gathering the above information one arrives at the following theorem.
\begin{thm} For $g>1$ and $d=3g-4+n$, 
the rank of $\kappa^{d}_c(\Mgnbar)$ is equal to
\begin{displaymath}
\left\lceil \frac{(n+1)(g+1)}{2}\right\rceil -1.
\end{displaymath}
\end{thm}
\begin{proof}
For $d=3g-4+n$, the lower bound on the rank is given by
the sum of the rank of the above $11\times 11$ 
matrix, plus the following sum 
\begin{displaymath}
\sum_{\substack{\q=(a_1\geq a_2>0)\\
5\leq a_2\leq (3g-4+n)/2}}r_f(\q;g,n).
\end{displaymath}
The above computations may then be used to compute this lower bound explicitly.
Let $\mathrm{row}(d;g,n)$ denote the number of rows in $R(d;g,n)$.
For $g>1$ we thus have
\begin{equation}\label{eq:rank-inequality}
\begin{split}
R_f(d;g,n)&\geq 11+ \sum_{\substack{\q=(a_1\geq a_2>4)}}r_f(\q;g,n)\\
%&\geq 11+\left(\sum_{\substack{\q=(a_1\geq a_2>4)}}{\mathrm{row}}_f(\q;g,n)\right)
%-E(g,n)\\
&=\mathrm{row}(d;g,n)=\left\lceil\frac{(n+1)(g+1)}{2}\right\rceil -1.
\end{split}
\end{equation}
Since $R_f(d;g,n)$ can not be larger than $\mathrm{row}(d;g,n)$, the 
inequality in (\ref{eq:rank-inequality}) is in fact an equality.
\end{proof}

\section{The asymptotic behaviour of the ranks }
The behaviour of the rank of $\kappa_c^d(\Mgnbar)$ 
for arbitrary values of $g,n$ and $d$ 
seems to be more complicated. Table~\ref{table:codim} illustrates the computations 
for genus $1,2,3$ in codimensions $2,3,4,5$ and $6$ when the number $n$ of the 
marked points is less than or equal to $10$. The second author has a computer 
program for computing the relevant kappa integrals, as well as the rank of the matrix
$R(d;g,n)$. Computations beyond these 
tables require large memory and are relatively time consuming even over very 
fast computers.  \\
\begin{table}
\caption{The tables illustrate the rank of $\kring$ in codimensions
$2,3,4,5$ and $6$ respectively.}\label{table:codim}
\vspace{4mm}
\begin{tabular}{|l|*{10}{c|}}\hline
\multicolumn{11}{|c|}{Codimension=2}\\\hline
\backslashbox{Genus}{Points}& 1 & 2 & 3 & 4 & 5 & 6 & 7 & 8 & 9 & 10 \\\hline
0 & 0 & 0 & 0 & 0 & 1 & 1 & 2 & 3 & 4 & 5 \\\hline
1 & 0 & 1 & 1 & 2 & 3 & 5 & 7 & 10 & 13 & 17 \\\hline
2 & 2 & 3 & 5 & 7 & 11 & 15 & 21 & 28 & 36 & 45 \\\hline
\end{tabular}

\vspace{7mm}
\begin{tabular}{|l|*{10}{c|}}\hline
\multicolumn{11}{|c|}{Codimension=3}\\\hline
\backslashbox{Genus}{Points}& 1 & 2 & 3 & 4 & 5 & 6 & 7 & 8 & 9 & 10 \\\hline
0 & 0 & 0 & 0 & 0 & 0 & 1 & 1 & 2 & 3 & 5 \\\hline
1 & 0 & 0 & 1 & 1 & 2 & 3 & 5 & 7 & 11 & 15 \\\hline
2 & 1 & 2 & 3 & 5 & 7 & 11 & 15 & 22 & 30 & 42 \\\hline
\end{tabular}

\vspace{7mm}
\begin{tabular}{|l|*{10}{c|}}\hline
\multicolumn{11}{|c|}{Codimension=4}\\\hline
\backslashbox{Genus}{Points}& 1 & 2 & 3 & 4 & 5 & 6 & 7 & 8 & 9 & 10 \\\hline
0 & 0 & 0 & 0 & 0 & 0 & 0 & 1 & 1 & 2 & 3 \\\hline
1 & 0 & 0 & 0 & 1 & 1 & 2 & 3 & 5 & 7 & 11 \\\hline
2 & 1 & 1 & 2 & 3 & 5 & 7 & 11 & 15 & 22 & 30 \\\hline
\end{tabular}

\vspace{7mm}
\begin{tabular}{|l|*{10}{c|}}\hline
\multicolumn{11}{|c|}{Codimension=5}\\\hline
\backslashbox{Genus}{Points}& 1 & 2 & 3 & 4 & 5 & 6 & 7 & 8 & 9 & 10 \\\hline
0 & 0 & 0 & 0 & 0 & 0 & 0 & 0 & 1 & 1 & 2 \\\hline
1 & 0 & 0 & 0 & 0 & 1 & 1 & 2 & 3 & 5 & 7 \\\hline
2 & 0 & 1 & 1 & 2 & 3 & 5 & 7 & 11 & 15 & 22 \\\hline
\end{tabular}

\vspace{7mm}
\begin{tabular}{|l|*{10}{c|}}\hline
\multicolumn{11}{|c|}{Codimension=2}\\\hline
\backslashbox{Genus}{Points}& 1 & 2 & 3 & 4 & 5 & 6 & 7 & 8 & 9 & 10 \\\hline
0 & 0 & 0 & 0 & 0 & 0 & 0 & 0 & 0 & 1 & 1\\\hline
1 & 0 & 0 & 0 & 0 & 0 & 1 & 1 & 2 & 3 & 5\\\hline
2 & 0 & 0 & 1 & 1 & 2 & 3 & 5 & 7 & 11 & 15 \\\hline
\end{tabular}
\end{table}
We apply the strategy of the previous section, and study the 
asymptotic behaviour of the rank of  $\kappa^d_c(\Mgnbar)$ instead, 
when the genus $g$ and the codimension $e=3g-3+n-d$ are fixed. 
The number of elements in $P(d,e+1)$, as $d$ grows large,
is asymptotic to ${d+e\choose e}/(e+1)!$.  
The number of rows in the 
matrix $R(d;g,n)$, i.e. the number of elements in $\QQ(d;g,n)$, 
is thus asymptotic to either of  
$$|\PP(d,e+1)|{g+e \choose e}\ \ \ \text{and}\ \ \ 
\frac{{n+e\choose e}{g+e\choose e}}{(e+1)!}.$$ 

Set $h=g+2$.
As $n$ grows large, the asymptotic growth of $\PP(d,e+1)$ is the same as the 
growth of the number 
of partitions $\q=(a_0\geq a_1\geq ...\geq a_e)$ satisfying 
$$a_e> 2eh,\ \ \ \ a_i\leq a_{i-1}-2h\ \ \ \ \text{for }i=1,...,e-1.$$
Denote the set of all such partitions by $\PP_h(d;e+1)$.\\

For $\q=(a_0\geq a_1\geq \hdots \geq a_e>0)\in \PP_h(d;e+1)$ 
let $P(\q)\subset \PP(d,2e+1)$ denote the set of partitions 
\begin{displaymath}
\q'=(a_0>a_1-b_1>...,a_e-b_e>b_e>b_{e-1}>\hdots >b_1),\ \ \ (i-1)h<b_{i}\leq ih.
\end{displaymath}
For every $\q'\in P(\q)$ define $f(\q')=\q$. This gives a function 
$$f:P=\bigcup_{\q\in \PP_h(d;e+1)}P(\q)\lra \PP_h(d;e+1).$$
Equip $\PP_h(d;e+1)$ with the lexicographic order, setting $\q=(a_0> ...> a_e>0)$
less than $\q=(a_0'> ...> a_e'>0)$ if there is some $i\geq 0$ such that 
$a_j=a_j'$ for $j=0,...,i-1$ and $a_i<a_i'$ (while $P$ is partially ordered 
with $\lhd$).\\

Although $f$ is not a fine assignment and the second condition in 
Definition~\ref{def:fine-assignment} may fail, 
it differs from a fine assignment in a controllable 
way, as will be discussed below. Let 
$$\q'=(a_0,a_1-b_1,...,a_e-b_e,b_e,...,b_1)$$ and 
suppose that $\q''\in \PP_h(d;e+1)$ refines $\q'$ while $\q''<\q=(a_0,...,a_e)$. Then 
\begin{displaymath}
\q''=(a_0\geq a_1\geq ...\geq a_{i-1}\geq a_i-b_i \geq c_{i+1}\geq ...\geq c_e>0)
\end{displaymath} 
for some $i>0$ and some positive integers $c_{i+1},...,c_e$. 
Associated with every $\q, \q'$ and  $\q''$ as above, 
and every $\p\in\QQ(\q'';g,n)$ we put  
$\frac{1}{\Lambda(\p)}\langle \q',\p\rangle$ in a matrix 
$E_i(d;g,n)$ as the entry corresponding to the row indexed by $\p$ and the column 
indexed by $\q'$. For $i=1,...,e$ the rows of the matrix $E_i(d;g,n)$ are labelled 
by 
$$Q=\bigcup_{\q\in\PP_h(d;e)}\QQ(\q;g,n),$$ 
while its columns are labelled by $P$.\\
 
The sub-matrix $S(d;g,n)$ of 
$R(d;g,n)$ determined by the columns corresponding to $P\subset \PP(d)$ and the rows 
corresponding to $Q\subset \QQ(d;g,n)$ is thus a sum
\begin{displaymath}
S(d;g,n)=T(d;g,n)+\sum_{i=1}^eE_i(d;g,n),
\end{displaymath}
where $T(d;g,n)$ is an upper triangular matrix with respect to the total 
order $<$. Lemma~\ref{lem:rank-general} implies that 
\begin{displaymath}
\rank(T(d;g,n))\geq \sum_{\q\in \PP(d;e)}r_f(\q;g,n)
\end{displaymath}

\begin{prop}\label{prop:asymptotic-lemma}
There is a subset $A\subset \PP_h(d;e+1)$ of size $\frac{(eh)^2}{2}|\PP(d;e+1)|$ 
such that for every $\q\in \PP_{h}(d;e+1)\setminus A$
$$r_f(\q;g,n)={g+e\choose e}.$$
\end{prop} 
\begin{proof}
Consider the matrix $S(\q;g,n)$ (containing $R_f(\q;g,n)$ as a sub-matrix) 
whose rows are in correspondence with all 
assignments $(g_0,g_1,...,g_e)$ to the $(e+1)$ components of the combinatorial
cycle, without any restriction on their sum, i.e. we consider all tuples 
$\p=((m_0,g_0),...,(m_e,g_e))$ such that $a_i=3g_i-3+m_i$ and $0\leq g_i\leq g$. 
The matrix $R_f(\q;g,n)$ is a sub-matrix of $S(\q;g,n)$ while they both have the 
same number of columns. As a result, if the row rank of 
$S(\q;g,n)$ is full, so is the row rank of $R(\q;g,n)$.\\

Let $M_k(m;g)$ denote the sub-matrix of $N(m;g)$ which consists of the columns 
corresponding to the values $j=k,k+1,...,k+g$.
The matrix $S(\q;g,n)$ has at least the same rank as the matrix 
$N_1\otimes N_2\otimes ...\otimes N_e$,
where $N_1=M(m;g)$ and for $i>1$ 
$N_i$ is the matrix $M_{ih-h+1}(a_i+2;g)$.  The determinant 
$d^k_g(m)=\Det(M_{k}(m;g))$ is a polynomial with 
$$\deg\left(d^k_g(m)\right)\leq \frac{(g+1)(2k+g)}{2}$$
by Lemma~\ref{lem:String}. Moreover, Proposition~\ref{prop:symmetry-1}
implies that 
\begin{displaymath}
d^k_g(k+g)=(-1)^{g\choose 2}\Det\left(\begin{array}{cccc}
n_0(0,k+g)&n_1(0,k+g)&\hdots &n_g(0,k+g)\\
n_0(1,k+g)&n_1(1,k+g)&\hdots &n_g(1,k+g)\\
\vdots       &\vdots       &\ddots &\vdots\\
n_0(0,k+g)&n_1(0,k+g)&\hdots &n_g(0,k+g)\\
\end{array}\right).
\end{displaymath}
Since the degree of $n_i(h,m)$ is $i$, the right-hand-side of the above equality 
is a Van-der-Monde matrix, and the determinant is non-zero. The polynomial 
$d^k_g(m)$ is thus non-trivial and has at most $\frac{(g+1)(2k+g)}{2}$ roots.
Let $A_g^k$ denote the set of integer roots of $d_g^k(m)$. 
Thus, $N_i$ is invertible unless $a_i+2$ belongs $A_g^{ih-h+1}$, and
$$|A_g^{ih-h+1}|=\frac{(g+1)(2ih-2h+2+g)}{2}\leq (i-\frac{1}{2})h^2.$$ 
The set of partitions $(a_0>a_1>...>a_e)\in\PP_h(d;e+1)$ such that 
$a_i\in A_g^{ih-h+1}$ is at most of size 
$$|\PP_h(d;e)||A_g^{ih-h+1}|\leq |\PP(d;e)|h^2(i-\frac{1}{2}).$$
Consequently, for $\q$ outside a set of size
$$\sum_{i=2}^e(i-\frac{1}{2})|\PP(d;e)|h^2=\frac{(eh)^2}{2}|\PP(d;e)|$$
every $N_i$ is a full-rank matrix and the rank of $R_f(\q;g,n)$ is  equal 
to the number of its rows, i.e. ${g+e\choose e}$.
\end{proof} 

Since $|\PP(d;e)|$ is asymptotic to ${n+e-1\choose e-1}/e!$,  
Proposition~\ref{prop:asymptotic-lemma} implies that the rank of $T(d;g,n)$ is 
asymptotic to 
$$|\PP_h(d;e+1)|{g+e\choose e}\simeq \frac{{n+e\choose e}{g+e\choose e}}{(e+1)!}.$$
In order to complete a computation of the asymptotic behaviour of $r(d;g,n)$
it suffices to study the difference between the rank of $T(d;g,n)$ and the rank of 
$S(d;g,n)$.\\

\begin{prop}\label{prop:asymptotic-lemma-2}
With the above notation
$$\lim_{n\ra \infty}\frac{\rank(E_i(d;g,n))}{{n+e\choose e}}=0.$$
\end{prop}
\begin{proof}
Define a function $f_i:Q\subset \QQ(d;g,n)\ra \QQ(d;g,n-1)$ as follows. Let 
$\p=\{(g_i,m_i)\}_{i=0}^{e}\in Q$ be a multi-set with $3g_i+m_i>3g_j+m_j$ if 
$i<j$. Define 
\begin{displaymath}
\begin{split}
&f_i:Q\subset \QQ(d;g,n)\ra \QQ(d;g,n-1)\\
&f_i(\p):=\left((g_0+g_i,m_0+m_i-3),(g_1,m_1),...,\widehat{(g_{i},m_{i})},\hdots,(g_e,m_e)
\right),
\end{split}
\end{displaymath}
where the hat over $(g_i,m_i)$ means that it is omitted from the sequence.
Let $Q_i=f_i(Q)\subset \QQ(d;g,n-1)$. Suppose that  
$$\p=\Big\{(g_i,m_i)\Big\}_{i=0}^{e+1}, \p'=\Big\{(g_i',m_i')\Big\}_{i=0}^{e+1},\ \ 
\text{and }\   f_i(\p)=f_i(\p').$$
 Furthermore, let 
$3g_j+m_j\geq 3g_k+m_k$ and $3g_j'+m_j'\geq 3g_k'+m_k'$ if $j<k$. For every 
$\q=(a_0,a_1-b_1,...,a_e-b_e,b_e,...,b_1)$ 
such that $\langle \q,\p\rangle$ is non-zero in $E_i(d;g,n)$,
$$\q(\p)=(a_0\geq a_1\geq ...\geq a_{i-1}\geq a_i-b_i\geq c_{i+1}\geq ...\geq c_e>0),$$
for some integers $c_{i+1},...,c_e$. Let 
\begin{displaymath}
\ov\q=\left (a_{i+1}-b_{i+1},...,a_e-b_e, b_e, ..., b_{i}\right),
\ \ \ \ov\p=\left((g_{i+1},m_{i+1}),\hdots (g_e,m_e)\right).
\end{displaymath} 
Then, 
\begin{displaymath}
\begin{split}
\langle \q,\p\rangle&=\frac{1}{24^{g_0+g_i}\times 
g_0!\times g_i!}
\left(\prod_{j=1}^{i-1}\int_{\Modbar_{g_j,m_j+2}}\psi_1^{b_j+1}\psi_2^{a_j-b_j+1}\right)
\left\langle \ov\q,\ov\p
\right\rangle\\
&=\langle \q',\p\rangle.
\end{split}
\end{displaymath} 
Thus, the rows in $E_i(d;g,n)$
which correspond to $\p$ and $\p'$ are identical. Consequently, 
the rank of $E_i(d;g,n)$ is 
bounded above by $|Q_i|$, which is in turn less than or equal to the cardinality 
of $\QQ(d;g,n-1)$. But the latter cardinality is asymptotic to 
\begin{displaymath}
\frac{{g+e-1\choose e-1}{n+e\choose e-1}}{e!}.
\end{displaymath}
The proposition follows immediately.
\end{proof}

\begin{thm}
The  rank of 
the kappa ring $\kappa_c^*(\Mgnbar)$ in codimension $e$, as the number 
$n$ of the marked points becomes large,
is asymptotic to 
$$\frac{{n+e \choose e}.{g+e\choose e}}{(e+1)!}$$
\end{thm}
\begin{proof}
 Proposition~\ref{prop:asymptotic-lemma}
implies that  asymptotically, the rank is greater than or equal to 
$$|\PP_h(d;e+1)|{g+e\choose e}-\sum_{i=1}^e \rank(E_i(d;g,n)).$$
By Proposition~\ref{prop:asymptotic-lemma-2}, the matrices $E_i(d;g,n)$ do not 
change the asymptotic, and $r(3g-3+n-e;g,n)$ is asymptotically greater than or equal to 
$$\frac{{n+e \choose e}.{g+e\choose e}}{(e+1)!}.$$
 Since  the number of rows in
$R(d;g,n)$ follows the same asymptotic behaviour the proof is complete.
\end{proof}
% ----------------------------------------------------------------

\end{document}